\documentclass[12pt]{amsart}
\usepackage{amssymb,amsmath,amsthm,amsfonts,times,hyperref,mathrsfs,multirow}
\newtheorem{theorem}{Theorem}[section]
\newtheorem{lemma}[theorem]{Lemma}
\newtheorem{cor}[theorem]{Corollary}

\newtheorem{drconj}[theorem]{Dyson's Rank Conjecture}
\newtheorem{prop}[theorem]{Proposition}

\theoremstyle{definition}
\newtheorem{definition}[theorem]{Definition}
\newtheorem{example}[theorem]{Example}

\theoremstyle{remark}
\newtheorem*{remark}{Remark}

\numberwithin{equation}{section}

\newcommand{\SL}{\mbox{SL}}
\newcommand{\GL}{\mbox{GL}}

\newcommand{\Parans}[1]{\left(#1\right)}

\newcommand\leg[2]{\genfrac{(}{)}{}{}{#1}{#2}} 

\newcommand\Mac[3]{M\left(\frac{#1}{#2};#3\right)}
\newcommand\Nac[3]{N\left(\frac{#1}{#2};#3\right)}

\newcommand\Lpar[1]{\left(#1\right)}
\newcommand\Mell[3]{\mathcal{M}\left(\frac{#1}{#2};#3\right)}
\newcommand\Nell[3]{\mathcal{N}\left(\frac{#1}{#2};#3\right)}
\newcommand\Fell[4]{\mathcal{F}_{#1}\left(\frac{#2}{#3};#4\right)}
\newcommand\SFell[4]{\mathcal{F}_{#1}^{*}\left(\frac{#2}{#3};#4\right)}

\newcommand\Gell[4]{\mathcal{G}_{#1}\left(\frac{#2}{#3};#4\right)}
\newcommand\Jpz[1]{\mathcal{J}\left(\frac{1}{p};{#1}\right)}
\newcommand\Jdpz[1]{\mathcal{J}\left(\frac{d}{p};{#1}\right)}

\newcommand\JSdpz[1]{\mathcal{J}^{*}\left(\frac{d}{p};{#1}\right)}
\newcommand\Rpz[1]{\mathcal{R}_p\left(#1\right)}
\newcommand\Rpzb[2]{\mathcal{R}_{#1}\left(#2\right)}

\newcommand\THA[3]{\Theta_1\Lpar{\frac{#1}{#2};#3}}
\newcommand\twidit[1]{\overset {\text{\lower 3pt\hbox{$\sim$}}}#1}
\newcommand\dtwidit[1]{\overset {\text{\lower 6pt\hbox{$\sim$}}}#1}
\newcommand\Wtwid{\overset {\text{\lower 3pt\hbox{$\sim$}}}W}
\newcommand\gtwid{\overset {\text{\lower 3pt\hbox{$\sim$}}}g}
\newcommand\ttwid{\overset {\text{\lower 3pt\hbox{$\sim$}}}\theta}
\newcommand\mutwid{\overset {\text{\lower 3pt\hbox{$\sim$}}}\mu}

\newcommand\zcon{\overline{z}}

\newcommand\AMat{\begin{pmatrix} a & b \\ c & d \end{pmatrix}}

\newcommand\SMat{\begin{pmatrix} 0 & -1 \\ 1 & 0 \end{pmatrix}}
\newcommand\TMat{\begin{pmatrix} 1 & 1 \\ 0 & 1 \end{pmatrix}}

\newcommand\FL[1]{\left\lfloor#1\right\rfloor}
\newcommand\CL[1]{\left\lceil#1\right\rceil}

\newcommand\strokeb[2]{\,\left\arrowvert\,\left[#1\right]_#2\right.}
\newcommand\stroke[3]{#1\,\left\arrowvert\,\left[#2\right]_{#3}\right.}

\newcommand\ord{\mbox{ord}}         
\newcommand\ORD{\mbox{ORD}}         
\newcommand\hord{\mbox{ord}_{\mbox{\scriptsize holo}}}         

\allowdisplaybreaks
\newcommand\mylabel[1]{\label{#1}}
\newcommand\mybibitem[1]{\bibitem{#1}}
\newcommand\thm[1]{\ref{thm:#1}}
\newcommand\lem[1]{\ref{lem:#1}}
\newcommand\corol[1]{\ref{cor:#1}}

\newcommand\propo[1]{\ref{propo:#1}}
\newcommand\eqn[1]{(\ref{eq:#1})}

\newcommand\refdef[1]{\ref{def:#1}}
\newcommand\subsect[1]{\ref{subsec:#1}}

\begin{document}
\newcommand{\beqs}{\begin{equation*}}
\newcommand{\eeqs}{\end{equation*}}
\newcommand{\beq}{\begin{equation}}
\newcommand{\eeq}{\end{equation}}
\renewcommand{\MR}[1]{\href{http://www.ams.org/mathscinet-getitem?mr={#1}}{MR{#1}}}
\title[Dyson rank function symmetries]
{New symmetries for Dyson's rank function}

\author{F. G. Garvan}
\address{Department of Mathematics, University of Florida, Gainesville,
FL 32611-8105}
\email{fgarvan@ufl.edu}
\author{Rishabh Sarma}
\address{Department of Mathematics, University of Florida, Gainesville,
FL 32611-8105}
\email{rishabh.sarma@ufl.edu}


\date{\today}                   



\begin{abstract}
At the 1987 Ramanujan Centenary meeting Dyson asked for a coherent 
group-theoretical structure for Ramanujan's mock theta functions 
analogous to Hecke's theory of modular forms.  Many of Ramanujan's mock theta functions can be written in terms of $R(\zeta_p,q)$, where $R(z,q)$ is  the two-variable generating function of Dyson's rank function and  $\zeta_p$ is a primitive $p$-th root of unity. 
In his lost notebook Ramanujan gives the $5$-dissection of $R(\zeta_5,q)$. This result is related to Dyson's famous rank conjecture which was proved by Atkin and Swinnerton-Dyer. In 2016 the first author showed that there is an analogous result for the $p$-dissection of $R(\zeta_p,q)$ when $p$ is any prime greater than $3$, by extending work of Bringmann and Ono, and Ahlgren and Treneer. It was also shown how the group $\Gamma_1(p)$ acts on the elements of the $p$-dissection of $R(\zeta_p,q)$. We extend this to the group $\Gamma_0(p)$, thus revealing new and surprising symmetries for Dyson's rank function.
\end{abstract}

\maketitle
\section{Introduction}
\mylabel{sec:intro}

Let $p(n)$ denote the number of partitions of $n$. The following are
Ramanujan's famous partition congruences:
\begin{align*}
p(5n+4) &\equiv 0 \pmod{5},\\
p(7n+5) &\equiv 0 \pmod{7},\\
p(11n+6) &\equiv 0 \pmod{11}.
\end{align*}
In 1944, Dyson \cite{Dy44} sought a simple combinatorial explanation for
these congruences. He defined the rank of a partition as the largest part minus 
the number of parts and conjectured that the rank mod $5$ divided the partitions 
of $5n+4$ into $5$ equal classes and that the rank mod $7$ divided the partitions of 
$7n+5$ into $7$ equal classes. His mod $5$ and $7$ rank conjectures
were proved by Atkin and Swinnerton-Dyer \cite{At-Sw}. 

Let $N(m,n)$ denote the number of partitions of $n$ with rank $m$. We let 
$R(z,q)$ denote the two-variable generating function for the Dyson rank function 
so that
$$
R(z,q) = \sum_{n=0}^\infty \sum_m N(m,n)\,z^m\,q^n.
$$
It is the symmetry of the rank function $R(z,q)$, when $z$
is a root of unity,  that we study in this paper. Many of Ramanujan's mock theta
functions can be written in such a form. For example,
Ramanujan's third order mock theta function $f(q)$ can be written
$$                        
f(q) = 1 + \sum_{n\ge1} \frac{q^{n^2}}{(1+q)^2(1+q^2)^2 \cdots (1+q^n)^2} = R(-1,q).
$$
In fact, it was Dyson \cite[p.20]{Dy1988}, 
who originally called for the study of such symmetry.

\begin{quote}
The mock theta-functions give us tantalizing hints of a grand synthesis
still to be discovered. Somehow it should be possible to build them
into a coherent group-theoretical structure, analogous to the structure
of modular forms which Hecke built around the old theta-functions of
Jacobi. This remains a challenge for the future.

\smallskip

\hskip2in Freeman Dyson, 1987

\hskip2in Ramanujan Centenary Conference\\
\end{quote}

Many authors have taken up Dyson's challenge. In this paper 
we extend the previous work of Ahlgren and Treneer \cite{Ah-Tr08},
Bringmann and Ono \cite{Br-On10} and the first author \cite{Ga19a}.

We have the following identities for the rank generating function $R(z,q)$:
\begin{align}
R(z,q) 
&= 1 + \sum_{n=1}^\infty \frac{q^{n^2}}{(zq,z^{-1}q;q)_n}
\label{eq:Rzqid1}\\
&= \frac{1}{(q;q)_\infty} \Lpar{1 + \sum_{n=1}^\infty 
   \frac{(-1)^n (1 + q^n)(1-z)(1-z^{-1})}
   {(1 - zq^n)(1 - z^{-1}q^n)} \, q^{\frac{1}{2}n(3n+1)}}.
\label{eq:Rzqid2}
\end{align}
See \cite[Eqs (7.2), (7.6)]{Ga88a}.
\\\\
Here and throughout this paper we use the standard $q$-notation:
\begin{align*}
(a;q)_{\infty} &= \prod_{k=0}^\infty (1-aq^k),
\\
(a;q)_{n} &= \frac{(a;q)_{\infty}}{(aq^n;q)_{\infty}},
\\
(a_1,a_2,\dots,a_j;q)_{\infty} 
&= (a_1;q)_{\infty}(a_2;q)_{\infty}\dots(a_j;q)_{\infty},
\\
(a_1,a_2,\dots,a_j;q)_{n} 
&=
(a_1;q)_{n}(a_2;q)_{n}\dots(a_j;q)_{n}.
\end{align*}
\\
Let $N(r,t,n)$ denote the number of partitions of $n$ with rank
congruent to $r$ mod $t$,
and let
$\zeta_p=\exp(2\pi i/p)$. Then
\beq
R(\zeta_p,q) = \sum_{n=0}^\infty \Lpar{\sum_{k=0}^{p-1} N(k,p,n)\,\zeta_p^k}\,q^n.
\label{eq:Rzetaid}
\eeq
Dyson's rank conjectures may be restated

\begin{drconj}[1944]
\label{conj:DRC}
For all nonnegative integers $n$,
\begin{align}
N(0,5,5n+4) &= N(1,5,5n+4) = \cdots  = N(4,5,5n+4) = \tfrac{1}{5}p(5n + 4),
\label{eq:Dysonconj5}\\
N(0,7,7n+5) &= N(1,7,7n+5) = \cdots  = N(6,7,7n+5) = \tfrac{1}{7}p(7n + 5).    
\label{eq:Dysonconj7}
\end{align}
\end{drconj}

As noted in \cite{Ga88a}, \cite{Ga88b}, it can be shown that
Dyson's mod $5$ rank conjecture \eqn{Dysonconj5} follows from an identity
in Ramanujan's Lost Notebook \cite[p.20]{Ra1988}, \cite[Eq. (2.1.17)]{An-Be-RLNIII}.
We let $\zeta_5$ be a primitive $5$th root of unity. Then the
following is Ramanujan's identity.
\begin{align}
R(\zeta_5,q) &= A(q^5) + (\zeta_5 + \zeta_5^{-1}-2)\,\phi(q^5) + q\,B(q^5) + 
(\zeta_5+\zeta_5^{-1})\,q^2\,C(q^5)
\label{eq:Ramid5}\\
&\quad 
- (\zeta_5+\zeta_5^{-1})\,q^3\left\{D(q^5) - (\zeta_5^2 + \zeta_5^{-2}  - 2)\frac{\psi(q^5)}{q^5}
\right\},
\nonumber
\end{align}
where
$$
A(q) = \frac{(q^2,q^3,q^5;q^5)_\infty}{(q,q^4;q^5)_\infty^2},\,
B(q) = \frac{(q^5;q^5)_\infty}{(q,q^4;q^5)_\infty},\,
C(q) = \frac{(q^5;q^5)_\infty}{(q^2,q^3;q^5)_\infty},\,
D(q) = \frac{(q,q^4,q^5;q^5)_\infty}{(q^2,q^3;q^5)_\infty^2},
$$
and
$$
\phi(q) = -1 + \sum_{n=0}^\infty \frac{q^{5n^2}}{(q;q^5)_{n+1} (q^4;q^5)_n},\qquad
\psi(q) =  -1 + 
\sum_{n=0}^\infty \frac{q^{5n^2}}{(q^2;q^5)_{n+1} (q^3;q^5)_n}.
$$
By multiplying by an appropriate power of $q$ and
substituting $q=\exp(2\pi iz)$,
we recognize the functions $A(q)$, $B(q)$, $C(q)$, $D(q)$ 
as being modular forms.
In fact, we can rewrite Ramanujan's identity \eqn{Ramid5} in 
terms of generalized eta-products:
$$
\eta(z) := q^{\frac{1}{24}} \prod_{n=1}^\infty (1 - q^n),\qquad q=\exp(2\pi iz),
$$
and

\begin{equation}
\label{eq:Geta}
\eta_{N,k}(z) =q^{\frac{N}{2} P_2(k/N) }
\prod_{
       \substack{m>0 \\ m\equiv \pm k\pmod{N}}} (1 - q^m),
\end{equation}

where $z \in \mathfrak{h}$, $P_2(t) = \{t\}^2 - \{t\} + \tfrac16$ 
is the second periodic Bernoulli polynomial, 
and $\{t\}=t-\lfloor t\rfloor$ is the fractional part of $t$.
Here as in Robins \cite{Ro94}, $1\le N\nmid k$.
We have
\begin{align}
&q^{-\frac{1}{24}}\left(R(\zeta_5,q) 
 - (\zeta_5 + \zeta_5^4 - 2) \,\phi(q^5) 
 + (1+ 2 \zeta_5 + 2\zeta_5^4 ) \,q^{-2}\,\psi(q^5)\right)  
\label{eq:Ramid5V2}\\
&
= \frac{\eta(25z) \, \eta_{5,2}(5z)}{\eta_{5,1}(5z)^2}
+
\frac{\eta(25z)}{\eta_{5,1}(5z)}
+ (\zeta_5 + \zeta_5^4) \, 
\frac{\eta(25z)}{\eta_{5,2}(5z)}
- (\zeta_5 + \zeta_5^4) \, 
\frac{\eta(25z)\, \eta_{5,1}(5z)}{\eta_{5,2}(5z)^2}.
\nonumber
\end{align}
Equation \eqn{Ramid5}, or equivalently \eqn{Ramid5V2}, give the $5$-dissection of the $q$-series expansion of $R(\zeta_5,q)$. We observe that the function on the right side of \eqn{Ramid5V2} is a weakly holomorphic modular form (with multiplier) of weight $\tfrac{1}{2}$ on the group $\Gamma_0(25) \cap \Gamma_1(5)$.
\\\\
In \cite{Ga19a}, the first author was able to generalize
Ramanujan's result \eqn{Ramid5V2} to all primes $p>3$.

\begin{theorem}[{\cite[Theorem 1.2, p.202]{Ga19a}}]
\label{thm:mainJp0}
For $p>3$ prime and $1 \le a \le \tfrac{1}{2}(p-1)$ define
\beq
\Phi_{p,a}(q) 
:= 
\begin{cases}
\displaystyle
\sum_{n=0}^\infty \frac{q^{pn^2}} 
{(q^a;q^p)_{n+1} (q^{p-a};q^p)_n},
 & \mbox{if $0 < 6a < p$,}\\
\displaystyle
-1 + \sum_{n=0}^\infty \frac{q^{pn^2}} 
{(q^a;q^p)_{n+1} (q^{p-a};q^p)_n},
 & \mbox{if $p < 6a < 3p$,}\\
\end{cases}
\label{eq:phipadef}
\eeq
and
\begin{align}      
\Rpz{\zeta_p,z} 
&:= q^{-\frac{1}{24}} R(\zeta_p,q) 
- 
\chi_{12}(p) \,
\sum_{a=1}^{\frac{1}{2}(p-1)} (-1)^a 
 \,
  \left( \zeta_p^{3a + \frac{1}{2}(p+1)} + \zeta_p^{-3a - \frac{1}{2}(p+1)}  
  \right.
\label{eq:Rpdef}\\
&\hskip 2in \left.  - 
    \zeta_p^{3a + \frac{1}{2}(p-1)} - \zeta_p^{-3a - \frac{1}{2}(p-1)}\right)
  \,
q^{\tfrac{a}{2}(p-3a)-\tfrac{p^2}{24}}\,\Phi_{p,a}(q^p),
\nonumber               
\end{align}      
where
\beq
\chi_{12}(n) := \leg{12}{n} = 
\begin{cases}
1 & \mbox{if $n\equiv\pm1\pmod{12}$,}\\
-1 & \mbox{if $n\equiv\pm5\pmod{12}$,}\\
0 & \mbox{otherwise,}
\end{cases}
\label{eq:chi12}
\eeq
and $q=\exp(2\pi iz)$ with  $\Im(z)>0$.
Then the function 
$$
\eta(p^2 z) \, \Rpz{\zeta_p,z} 
$$
is a weakly holomorphic modular form of weight $1$ on the group
$\Gamma_0(p^2) \cap \Gamma_1(p)$.
\end{theorem}

In \cite{Ga19a}, the first author also considered the modularity of each element of the $p$-dissection of 
$\eta(p^2 z)\,\Rpz{\zeta_p,z}$. For example, we have

\begin{theorem}[{\cite[Corollary 1.3, p.202]{Ga19a}}]
\label{cor:Kp0}
Let $p > 3$ be prime and $s_p=\tfrac{1}{24}(p^2-1)$. Then the function 
$$
\prod_{n=1}^\infty (1-q^{pn})\, 
\sum_{n=\CL{\frac{1}{p}(s_p)}}^\infty \left(\sum_{k=0}^{p-1} N(k,p,pn -s_p)\,\zeta_p^k\right)q^n
$$                      
is a weakly holomorphic modular form of weight $1$ on the group
$\Gamma_1(p)$.
\end{theorem}

\begin{remark}
This result is an improvement of a theorem of Ahlgren and Treneer 
\cite[Theorem 1.6, p.271]{Ah-Tr08}.
\end{remark}

In this paper we improve the previous results of the first author for the other elements
of the $p$-dissection. These elements feature the functions 
$\mathcal{K}_{p,m}(\zeta_p^d,z)$.

\begin{definition}
\label{def:Kpm}
For $p>3$ prime,  $0 \le m \le p-1$ and $1 \le d \le p-1$ 
define $\mathcal{K}_{p,m}(\zeta_p^d,z)$ as follows :
\begin{enumerate}
\item[(i)]
For $m=0$ or $\leg{-24m}{p}=-1$ define 
\beq
\mathcal{K}_{p,m}(\zeta_p^d,z) := q^{m/p}\,\prod_{n=1}^\infty (1-q^{pn})\, 
\sum_{n=\CL{\frac{1}{p}(s_p-m)}}^\infty \left(\sum_{k=0}^{p-1} N(k,p,pn + m -s_p)\,\zeta_p^{kd}\right)q^n,
\label{eq:Kpm1prop}
\eeq
where $s_p=\frac{1}{24}(p^2-1)$, and $q=\exp(2\pi iz)$. 
\item[(ii)]
For $\leg{-24m}{p}=1$ define                                                  
\begin{align}
\mathcal{K}_{p,m}(\zeta_p^d,z) 
&:= q^{m/p}\,\prod_{n=1}^\infty (1-q^{pn})\, 
\Bigg(
\sum_{n=\CL{\frac{1}{p}(s_p-m)}}^\infty \left(\sum_{k=0}^{p-1} N(k,p,pn + m -s_p)\,\zeta_p^{kd}\right)q^n
\label{eq:Kpm2prop}
\\
& \quad            
- 4 \chi_{12}(p) \, (-1)^{a+d+1} \,
  \sin\Lpar{\frac{d\pi}{p}} \, 
  \sin\Lpar{\frac{6ad\pi}{p}} \, 
  q^{\frac{1}{p}( \frac{a}{2}(p - 3a) - m)}
  \,
  \Phi_{p,a}(q) 
  \Bigg),
\nonumber
\end{align}
where $1\le a \le \frac{1}{2}(p-1)$ has been chosen so that
$$
-24m \equiv \Lpar{6a}^2 \pmod{p}.
$$
\end{enumerate}
\end{definition}

In \cite{Ga19a}, the first author studied the action of the the group $\Gamma_1(p)$ on  $\mathcal{K}_{p,m}(\zeta_p^d,z)$ for $d=1$ and obtained :

\begin{theorem}\cite[Theorem 6.3, p.234]{Ga19a}
\label{thm:KpmthmG}
Suppose $p>3$ prime, $0 \le m \le p-1$.             
Then\\
\begin{enumerate}
\item[(i)]
$\mathcal{K}_{p,0}(\zeta_p,z)$ is a weakly holomorphic modular form of weight 
$1$ on
$\Gamma_1(p)$.
\item[(ii)]
If $1 \le m \le (p-1)$ then $\mathcal{K}_{p,m}(\zeta_p,z)$ is a 
weakly holomorphic modular form of weight $1$ on
$\Gamma(p)$.                           
In particular,
\beq
\stroke{\mathcal{K}_{p,m}(\zeta_p,z)}{B}{1}
= \exp\Lpar{\frac{2\pi ibm}{p}}\,\mathcal{K}_{p,m}(\zeta_p,z),
\label{eq:Kpmtrans}
\eeq
for $B=\AMat\in\Gamma_1(p)$.  
\end{enumerate}
\end{theorem}
\begin{remark}
In equation \eqn{Kpmtrans} we have used the stroke operator notation
defined in \eqn{strokedef}.
\end{remark}

In \cite{An-Be-RLNIII} Entry 2.1.5, we have a 7 dissection of the $q$-series expansion of $R(\zeta_7,q)$. Rewriting the equation in an equivalent form in terms of generalized eta-functions, we have
\begin{align}
&q^{-\frac{1}{24}}\Big(R(\zeta_7,q) 
 + (\zeta_7 + \zeta_7^6 - 2) \,\phi_{7,1}(q^7) 
 + (-\zeta_7^2 + \zeta_7^3 + \zeta_7^4 - \zeta_7^5 ) \,q^{-1}\,\phi_{7,2}(q^7)
\\
\nonumber
&\hspace{58mm}+(1+ 2 \zeta_7^2 + \zeta_7^3 + \zeta_7^4 + 2 \zeta_7^5 ) \,q^{-5}\,\phi_{7,3}(q^7)\Big)  
\label{eq:Ramid7V2}\\
&
= (-1+\zeta_7 + \zeta_7^6) \frac{\eta(49z) \, \eta_{7,3}(7z)}{\eta_{7,1}(7z)\, \eta_{7,2}(7z)}
+ \frac{\eta(49z)}{\eta_{7,1}(7z)}
+ (\zeta_7 + \zeta_7^6) \, \frac{\eta(49z) \, \eta_{7,2}(7z)}{\eta_{7,1}(7z)\, \eta_{7,3}(7z)}\\
&
+ (1+\zeta_7^2 + \zeta_7^5) \,  \frac{\eta(49z)}{\eta_{7,2}(7z)}
- (\zeta_7^2 + \zeta_7^5) \,  \frac{\eta(49z)}{\eta_{7,3}(7z)}
- (1+\zeta_7^3 + \zeta_7^4) \, \frac{\eta(49z) \, \eta_{7,1}(7z)}{\eta_{7,2}(7z)\, \eta_{7,3}(7z)}.
\nonumber
\end{align}
Multiplying both sides by $\eta(49z)$ and letting 
$q\rightarrow q^{\frac{1}{7}}$, we  find the elements 
 $\mathcal{K}_{7,m}(\zeta_7, z)$ of the $7-$dissection of 
 $\eta(49z) \, \Rpzb{7}{\zeta_p,z}$ in terms of generalized eta-products:
\begin{align*}
\mathcal{K}_{7,0}(\zeta_7, z)&=0,\\
\mathcal{K}_{7,1}(\zeta_7, z)&=- (1+\zeta_7^3 + \zeta_7^4) \, \frac{\eta(7z)^2 \, \eta_{7,1}(z)}{\eta_{7,2}(z)\, \eta_{7,3}(z)},\\
\mathcal{K}_{7,2}(\zeta_7, z)&=(-1+\zeta_7 + \zeta_7^6) \frac{\eta(7z)^2 \, \eta_{7,3}(z)}{\eta_{7,1}(z)\, \eta_{7,2}(z)},\\
\mathcal{K}_{7,3}(\zeta_7, z)&=\frac{\eta(7z)^2}{\eta_{7,1}(z)},\\
\mathcal{K}_{7,4}(\zeta_7, z)&=(\zeta_7 + \zeta_7^6) \, \frac{\eta(7z)^2 \, \eta_{7,2}(z)}{\eta_{7,1}(z)\, \eta_{7,3}(z)},\\
\mathcal{K}_{7,5}(\zeta_7, z)&=(1+\zeta_7^2 + \zeta_7^5) \,  \frac{\eta(7z)^2}{\eta_{7,2}(z)},\\
\mathcal{K}_{7,6}(\zeta_7, z)&=- (\zeta_7^2 + \zeta_7^5) \,  \frac{\eta(7z)^2}{\eta_{7,3}(z)}.
\end{align*}
\\
The cases $m=1,2$ and $4$ correspond to when $-24m$ is a quadratic residue modulo $7$ and $m=3,5$ and $6$ to when $-24m$ is a quadratic non-residue modulo $7$. We can clearly see symmetry among between the eta-quotients of the quadratic residue cases and likewise of the non-residue ones. We find that this behavior arises from the transformation of $\mathcal{K}_{p,m}(\zeta_p^d,z)$ under matrices in $\Gamma_0(p)$. This leads us to one of the main results of our paper :

\begin{theorem}
\label{thm:Kpmthm}
Suppose $p>3$ prime, $0 \le m \le p-1$, and $1 \le d \le p-1$.             
Then\\
\beq
\stroke{\mathcal{K}_{p,m}(\zeta_p, z)}{A}{1}
= \frac{\sin(\pi/p)}{\sin(d\pi/p)}\,(-1)^{d+1}\,
\exp\Lpar{\frac{2\pi imak}{p}}\,\mathcal{K}_{p,ma^2}(\zeta_p^d,z),
\label{eq:KpmAtrans}
\eeq
assuming $1 \le a,d \le (p-1)$ 
and
$$
A = \begin{pmatrix}
      a & k \\ p & d 
    \end{pmatrix}
\in \Gamma_0(p).
$$\\
\end{theorem}
Sometimes it is convenient to rewrite the generalized eta-products in terms of some theta-products of Biagioli. 
\\
\begin{definition}
\label{def:geneta}
Following Biagioli (see \cite[Eq.(2.8),p.277]{Bi89}), define
\beq
f_{N,\rho}(z) := q^{(N-2\rho)^2/(8N)}\,(q^\rho,q^{N-\rho},q^N;q^N)_\infty
\label{eq:fdef},
\eeq
for $N\ge1$ and $N\nmid\rho$.
We have corrected a misprint \cite[p.277]{Bi89} and 
\cite[Eq.(6.14),p.242]{Ga19a}.
\noindent
Then, for a vector $\overrightarrow{n}=(n_0,n_1,n_2,\cdots ,n_{\frac{1}{2}(p-1)}) \in \mathbb{Z}^{\frac{1}{2}(p+1)}$, define 
\beq
j(z)=j(p,\overrightarrow{n},z)=\eta(pz)^{n_0}\prod\limits_{k=1}^{\frac{1}{2}(p-1)}f_{p,k}(z)^{n_k}.
\label{eq:jdef}
\eeq 
\end{definition}
\noindent
We note that 
\beq 
f_{N,\rho}(z) = f_{N,N+\rho}(z) = f_{N,-\rho}(z),
\label{eq:propf1} 
\eeq 
and
\beq 
f_{N,\rho}(z) = \eta(Nz)\,\eta_{N,\rho}(z).
\label{eq:propf2}
\eeq
\\
We illustrate the theorem for $p=7$, with $m=1$ and $A = \begin{pmatrix}
      2 & 1 \\ 7 & 4 
    \end{pmatrix}
\in \Gamma_0(7)$.  
\\
Using \eqn{propf2}, we have $$\mathcal{K}_{7,1}(\zeta_7, z)=- (1+\zeta_7^3 + \zeta_7^4) \, \frac{\eta(7z)^3 \, f_{7,1}(z)}{f_{7,2}(z)\, f_{7,3}(z)}.$$ Then, using the Biagioli transformation identity and the transformation for $\eta(z)$ in \cite[Theorems 6.12 \& 6.14, p.243]{Ga19a}, we have
\begin{align*}
\stroke{\mathcal{K}_{7,1}(\zeta_7, z)}{A}{1}
&=\stroke{- (1+\zeta_7^3 + \zeta_7^4) \, \frac{\eta(7z)^3 \, f_{7,1}(z)}{f_{7,2}(z)\, f_{7,3}(z)}}{A}{1}\\
&=- (1+\zeta_7^3 + \zeta_7^4) \, \frac{\eta(7z)^3 \, f_{7,2}(z)}{f_{7,4}(z)\, f_{7,6}(z)}\frac{e^{\frac{2 \pi i}{7}}}{e^{\frac{8 \pi i}{7}}.e^{\frac{18 \pi i}{7}}}\\
&=- (1+\zeta_7^3 + \zeta_7^4) \, \frac{\eta(7z)^3 \, f_{7,2}(z)}{f_{7,3}(z)\, f_{7,1}(z)}.e^{\frac{-3 \pi i}{7}}\qquad\mbox{(by \eqn{propf1})}\\
&=\exp\Lpar{\frac{4\pi i}{7}}\,(1+\zeta_7^3 + \zeta_7^4) \, \frac{\eta(7z)^2 \, \eta_{7,2}(z)}{\eta_{7,3}(z)\, \eta_{7,1}(z)}\qquad\mbox{(by \eqn{propf2})}.\\
\end{align*}

It can be easily checked that $1+\zeta_7^3 + \zeta_7^4=-\frac{\sin(\pi/7)}{\sin(4\pi/7)}\,(\zeta_7^4 + \zeta_7^3)$. Therefore,

\begin{align*}
\stroke{\mathcal{K}_{7,1}(\zeta_7, z)}{A}{1}
&=\exp\Lpar{\frac{4\pi i}{7}}\,(1+\zeta_7^3 + \zeta_7^4) \, \frac{\eta(7z)^2 \, \eta_{7,2}(z)}{\eta_{7,3}(z)\, \eta_{7,1}(z)}\\
&=-\frac{\sin(\pi/7)}{\sin(4\pi/7)}\,
\exp\Lpar{\frac{4\pi i}{7}}\,(\zeta_7^4 + \zeta_7^3) \, \frac{\eta(7z)^2 \, \eta_{7,2}(z)}{\eta_{7,1}(z)\, \eta_{7,3}(z)}\\
&=-\frac{\sin(\pi/7)}{\sin(4\pi/7)}\,
\exp\Lpar{\frac{4\pi i}{7}}\,\mathcal{K}_{7,4}(\zeta_7^4,z).\\
\end{align*}
which agrees with the transformation in Theorem \thm{Kpmthm}.
\\\\
Our other main result of the paper concerns the symmetry of the zeta-coefficients in the identity for ${\mathcal{K}_{p,0}(\zeta_p, z)}$ in terms of generalized eta-functions. In {\cite[Section 6.4, p.242]{Ga19a}}, the first author found that \\
\beq
{\mathcal{K}_{11,0}(\zeta_{11}, z)}=(q^{11};q^{11})_\infty \,
\sum_{n=1}^\infty \Lpar{\sum_{k=0}^{10} N(k,11,11n-5)\,\zeta_{11}^k}q^n
=
\sum_{k=1}^5 c_{11,k}\, j_{11,k}(z),
\label{eq:rank11id}
\eeq
where
$$
j_{11,k}(z) = \frac{\eta(11z)^4}{\eta(z)^2} \, 
\frac{1}{\eta_{11,4k}(z)\,\eta_{11,5k}(z)^2},
$$
and
\begin{align*}
c_{11,1} &= 2\,{\zeta_{11}}^{9}+2\,{\zeta_{11}}^{8}+{\zeta_{11}}^{7}+{\zeta_{11}}^{4}+2\,{\zeta_{11}}^{3}+2
\,{\zeta_{11}}^{2}+1,\\
c_{11,2} & =-({\zeta_{11}}^{9}+{\zeta_{11}}^{8}+2\,{\zeta_{11}}^{7}+{\zeta_{11}}^{6}+{\zeta_{11}}^{5}+2\,{
\zeta_{11}}^{4}+{\zeta_{11}}^{3}+{\zeta_{11}}^{2}+1),\\
c_{11,3} &= 2\,{\zeta_{11}}^{8}+2\,{\zeta_{11}}^{7}+2\,{\zeta_{11}}^{4}+2\,{\zeta_{11}}^{3}+3,\\
c_{11,4} &= 4\,{\zeta_{11}}^{9}+{\zeta_{11}}^{8}+2\,{\zeta_{11}}^{7}+2\,{\zeta_{11}}^{6}
+2\,{\zeta_{11}}^{5}+2\,{\zeta_{11}}^{4}+{\zeta_{11}}^{3}+4\,{\zeta_{11}}^{2}+4,\\
c_{11,5} &=-({\zeta_{11}}^{9}+2\,{\zeta_{11}}^{8}-{\zeta_{11}}^{7}+2\,{\zeta_{11}}^{6}+2\,{\zeta_{11}}^{5}-
{\zeta_{11}}^{4}+2\,{\zeta_{11}}^{3}+{\zeta_{11}}^{2}+3).\\
\end{align*}
Our theorem reveals hidden symmetry in these coefficients. We find \\
\begin{align*}
c_{11,1} &= 1+2\,\left({\zeta_{11}}^{2}+{\zeta_{11}}^{9})+2\,({\zeta_{11}}^{3}+{\zeta_{11}}^{8})+({\zeta_{11}}^{4}+{\zeta_{11}}^{7}\right),
\\
c_{11,2} & =-\frac{\sin(\pi/11)}{\sin(2\pi/11)}(1+2\,({\zeta_{11}}^{4}+{\zeta_{11}}^{7})+2\,({\zeta_{11}}^{5}+{\zeta_{11}}^{6})+({\zeta_{11}}^{3}+{\zeta_{11}}^{8})),
\\
c_{11,3} &= -\frac{\sin(\pi/11)}{\sin(3\pi/11)}(1+2\,({\zeta_{11}}^{5}+{\zeta_{11}}^{6})+2\,({\zeta_{11}}^{2}+{\zeta_{11}}^{9})+({\zeta_{11}}^{}+{\zeta_{11}}^{10})),
\\
c_{11,4} & =\frac{\sin(\pi/11)}{\sin(4\pi/11)}(1+2\,({\zeta_{11}}^{3}+{\zeta_{11}}^{8})+2\,({\zeta_{11}}^{10}+{\zeta_{11}}^{})+({\zeta_{11}}^{6}+{\zeta_{11}}^{5})),
\\
c_{11,5} &= -\frac{\sin(\pi/11)}{\sin(5\pi/11)}(1+2\,({\zeta_{11}}^{}+{\zeta_{11}}^{10})+2\,({\zeta_{11}}^{7}+{\zeta_{11}}^{4})+({\zeta_{11}}^{2}+{\zeta_{11}}^{9})).
\\
\end{align*}
We clearly see symmetry in these coefficients.
\\\\
In general, our result implies certain symmetries for the coefficients in the identities for ${\mathcal{K}_{p,0}(\zeta_p, z)}$. These involve explicit modular forms in terms of the Biagioli theta functions $f_{N,\rho}(z)$ defined in \eqn{fdef}. In particular, they involve eta-quotients $j(p,\overrightarrow{n},z)$ defined in \eqn{jdef}.
\\\\
The case $m=0$ is quite special and leads us to our other major result. The exact form of this result is given later in Theorem \thm{coeffsymmKp0}.

\begin{theorem}
\label{thm:Kp0sym}
Let $p>3$ prime and the $t$ vectors $\overrightarrow{n_{\ell}}, 1\leq \ell \leq t$ and $j(z)$ be defined as in Definition \refdef{geneta}. Suppose
$$
\mathcal{K}_{p,0}(\zeta_p,z)
=\sum_{{\ell}=1}^t\sum_{r=1}^{\frac{1}{2}(p-1)}c_{p,r,{\ell}}(\zeta_p)
         j(p,\pi_r(\overrightarrow{n_{\ell}}),z),
$$ 
where for $1\leq r \leq \frac{1}{2}(p-1)$, 
$\pi_r$ is permutation on $\{1,2,\cdots,\frac{1}{2}(p-1)\}$ defined as $\pi_r(i)=i'$ 
where $ri'\equiv \pm i \pmod{p}$ and the functions 
$j(p,\pi_r(\overrightarrow{n_{\ell}}),z)$ are linearly independent (over $\mathbb{Q})$. 
Then
$$
c_{p,r,\ell}(\zeta_p)=\frac{\sin(\pi/p)}{\sin(r\pi/p)}\,w(r,p)c_{p,1,\ell}(\zeta_p^r),
$$ 
where $w(r,p)=\pm 1$. 
\end{theorem}
\begin{example}
\label{ex:Kp0symeg}
We illustrate the theorem for $p=11$. Here $t=1$. From \eqn{rank11id} we have
$$
\mathcal{K}_{11,0}(\zeta_{11},z)
=\sum_{r=1}^{5}c_{11,r,1}(\zeta_{11})
         j(11,\pi_r(\overrightarrow{n_{1}}),z),
$$
where
$$
\vec{n}_1 = (15, -2,-2 ,-2,-3,-4),
$$
$$
c_{11,r,1}(\zeta_{11}) 
= 1+2\,\left({\zeta_{11}}^{2}+{\zeta_{11}}^{9})
  +2\,({\zeta_{11}}^{3}+{\zeta_{11}}^{8})
   +({\zeta_{11}}^{4}+{\zeta_{11}}^{7}\right),
$$
$$
c_{11,r,1}(\zeta_{11})=\frac{\sin(\pi/11)}{\sin(r\pi/11)}\,w(r,11)
                        c_{11,1,1}(\zeta_p^r),\qquad (1 \le r \le 5),
$$ 
$$
w(1,11)=1, \quad
w(2,11)=-1, \quad
w(3,11)=-1, \quad
w(4,11)=1, \quad
w(5,11)=-1,
$$
and
$$
j(11,\pi_r(\overrightarrow{n_{1}}),z)
= \frac{\eta(11z)^4}{\eta(z)^2} \, 
\frac{1}{\eta_{11,4r}(z)\,\eta_{11,5r}(z)^2}=
\frac{\eta(11z)^{15}}
{f_{11,r}(z)^2 
f_{11,2r}(z)^2 
f_{11,3r}(z)^2 
f_{11,4r}(z)^3 
f_{11,5r}(z)^4 }.
$$
\end{example}

The paper is organized as follows. In Section 2, we review transformation  
results for theta functions and Maass forms. In Section 3, we give conditions 
for 
the modularity of the generalized eta-quotients which are the building blocks 
for our 
expressions for $\mathcal{K}_{p,m}(\zeta_{p},z)$ needed in a later section. 
In Section 4, we state and prove our main result on the symmetry of Dyson's 
rank function. This involves generalizing many earlier transformation results 
in \cite{Ga19a} and finding the transformation of 
$\mathcal{K}_{p,m}(\zeta_{p},z)$ under matrices in $\Gamma_0(p)$. In 
Section 5, we give another symmetry result, more precisely the symmetry 
among the cyclotomic coefficients in an expression for 
$\mathcal{K}_{p,0}(\zeta_{p},z)$ in terms of generalized eta-functions.  
Section 6 is devoted to calculating lower bounds for the orders of 
$\mathcal{K}_{p,m}(\zeta_{p},z)$ at the cusps of $\Gamma_1(p)$ which we 
utilize to prove identities in the subsequent section. 
In Section 7, we give an algorithm that uses the Valence formula
for proving generalized eta-quotient identities
for $\mathcal{K}_{p,m}(\zeta_{p},z)$. We illustrate the algorithm in
detail for $p=11$. We give explicit identities for $p=11$ and $p=13$.
For $p=17$ and $p=19$ we only give the form of the identities.


\section{Preliminary Definitions and Results}

In this section, we state the definitions and results from \cite{Ga19a}, which we will be using in the proof of our main theorem on the symmetry of Dyson's rank function.

\subsection{Dyson's rank function as a mock modular form}
Following \cite{Br-On10} and \cite{Ga19a} we define a number of functions.
Suppose $0 < a < c$ are integers, and assume throughout that $q:=\exp(2\pi iz)$.
We define
\begin{align*}
\Mac{a}{c}{z} &:= \frac{1}{(q;q)_\infty} \sum_{n=-\infty}^\infty 
                  \frac{(-1)^n q^{n+\frac{a}{c}}}
                       {1-q^{n+\frac{a}{c}}} \, q^{\frac{3}{2}n(n+1)}
\\
\Nac{a}{c}{z} &:= \frac{1}{(q;q)_\infty} \Lpar{1 + \sum_{n=1}^\infty 
                  \frac{(-1)^n (1 + q^n)\left(2 - 2\cos\left(\frac{2\pi a}{c}\right)\right)}
                       {1 - 2\cos\left(\frac{2\pi a}{c}\right)q^n + q^{2n}} 
                        \, q^{\frac{1}{2}n(3n+1)}}.
\end{align*}
We find
that
\begin{align}
\sum_{n=0}^\infty \frac{q^{cn^2}}{(q^a;q^c)_{n+1} (q^{c-a};q^c)_n}
&= 1 + \frac{q^a}{(q^c;q^c)_\infty} 
\sum_{n=-\infty}^\infty \frac{(-1)^n q^{\frac{3c}{2}n(n+1)}}{1-q^{cn+a}}
\label{eq:Macid}\\
&= 1 + q^a \, \Mac{a}{c}{cz}.
\nonumber
\end{align}
By \eqn{Rzqid2} we have
\beq
R(\zeta_c^a,q) = \Nac{a}{c}{z}.
\label{eq:RNacid}
\eeq
We also define
\begin{align}
\Nell{a}{c}{z} &:=\csc\Lpar{\frac{a\pi}{c}}\, q^{-\frac{1}{24}}\,\Nac{a}{c}{z},
\label{eq:Ndef}\\
\Mell{a}{c}{z} &:=2 q^{\frac{3a}{2c}\left(1-\frac{a}{c}\right) - \frac{1}{24}} \, \Mac{a}{c}{z}.
\label{eq:Mdef}
\end{align}
\subsection{Theta-function definitions}
We state some results of Shimura \cite{Sh} needed to derive a theta function identity which we use in the next subsection to relate two Maass forms of weight $\frac{1}{2}$.
\\\\
For integers $0\le k < N$ we define
$$
\ttwid(k,N;z) := \sum_{m=-\infty}^\infty (Nm+k) \exp\Lpar{\frac{\pi i z}{N}(Nm+k)^2}.
$$
We note that this corresponds to $\theta(z;k,N,N,P)$ in Shimura's 
notation \cite[Eq.(2.0), p.454]{Sh} (with $n=1$, $\nu=1$, and $P(x)=x$).
For integers $0\le a,b < c$ we define
$$
\Theta_1(a,b,c;z) := \zeta_{c^2}^{3ab}\, \zeta_{2c}^{-a}\,
\sum_{m=0}^{6c-1} (-1)^m \sin\Lpar{\frac{\pi}{3}(2m+1)}
\exp\Lpar{\frac{-2\pi i m a}{c}} \ttwid(2mc-6b+c,12c^2;z),
$$
and
\begin{align*}
\hspace{-17mm}\Theta_2(a,b,c;z) &:= \sum_{\ell=0}^{2c-1} \left( (-1)^\ell \, 
\exp\Lpar{\frac{-\pi ib}{c}(6\ell+1)} \, \ttwid(6c\ell + 6a + c,12c^2;z) \right.\\
&\qquad\qquad \left. + (-1)^\ell
\exp\Lpar{\frac{-\pi ib}{c}(6\ell-1)} \, \ttwid(6c\ell + 6a - c,12c^2;z) \right).
\end{align*}
An easy calculation gives
\begin{align}
\Theta_1(a,b,c;z) &=
6c\,\zeta_{c^2}^{3ab}\, \zeta_{2c}^{-a}\,
\sum_{n=-\infty}^\infty (-1)^n \Lpar{\frac{n}{3} + \frac{1}{6} - \frac{b}{c}}\,
\sin\Lpar{\frac{\pi}{3}(2n+1)}\,
\exp\Lpar{\frac{-2\pi i n a}{c}} 
\label{eq:Theta1abcid}\\
&\hskip 2in \times \exp\Lpar{3\pi i z\Lpar{\frac{n}{3} + \frac{1}{6} - \frac{b}{c}}^2}.
\nonumber
\end{align}
In addition we need to define
\beq
\THA{a}{c}{z} := \sum_{n=-\infty}^\infty (-1)^n (6n+1) \sin\Lpar{\frac{\pi a(6n+1)}{c}}
\exp\Lpar{3\pi iz\Lpar{n + \frac{1}{6}}^2}. 
\label{eq:Theta1acdef}
\eeq
This coincides with Bringmann and Ono's
function $\Theta\Lpar{\frac{a}{c};z}$ which is given in \cite[Eq.(1.6), p.423]{Br-On10}.
An easy calculation gives
$$
\THA{a}{c}{z}  = -\frac{i}{2c} \, \Theta_2(0,-a,c;z).
$$
\begin{prop}
\label{propo:theta1id}
Let $p > 3$ be prime. Then 
\beq
\THA{d}{p}{z}
=
(-1)^d\, \frac{2}{\sqrt{3}} \,\chi_{12}(p) \,
\sum_{a=1}^{\frac{1}{2}(p-1)} (-1)^a 
\sin\Lpar{\frac{6ad\pi}{p}} \,\Theta_1(0,-a,p;p^2z).
\label{eq:theta1id}
\eeq
\end{prop}

\begin{proof}
We consider two cases as in \cite[Theorem 5.1, p.225]{Ga19a}.

\subsubsection*{CASE 1} $p\equiv 1 \pmod{6}$. Let $p_1 = \frac{1}{6}(p-1)$
so that $6p_1+1=p$. We note that each integer $n$ satisfying $6n+1\not\equiv0
\pmod{p}$ can be written uniquely as
\begin{align*}
&(i) \qquad n=p(2pm+\ell_1) + a  + p_1, \quad
     \mbox{where $1\le a \le \frac{1}{2}(p-1)$, $0 \le \ell_1 < 2p$, and 
           $m \in \mathbb{Z}$,}\\
&\hspace{75mm}\mbox{or}\\
&(ii) \qquad n=p(-2pm-\ell_1) - a  + p_1, \quad
     \mbox{where $1\le a \le \frac{1}{2}(p-1)$, $1 \le \ell_1 \le 2p$, and 
           $m \in \mathbb{Z}$.}\\
\end{align*}
If $n=p(2pm+\ell_1) + a  + p_1$, then
\begin{align*}
&6n+1 = 12p^2m + 2p\ell + 6a+p,\qquad\mbox{where $\ell=3\ell_1$ and 
$0\le \ell < 6p$},
\\
&\sin\Lpar{\frac{\pi\,d\, (6n+1)}{p}} 
= \sin\Lpar{\frac{\pi\,d\, (12p^2m+2p\ell+6a+p)}{p}}
=\sin\Lpar{\frac{6ad\pi}{p}+\pi d}
\\
&\hspace{101mm}=(-1)^d\sin\Lpar{\frac{6ad\pi}{p}},
\\
&\sin\Lpar{\frac{\pi}{3}(2\ell+1)} = \frac{\sqrt{3}}{2}, \quad (-1)^n = (-1)^{\ell + a + p_1}.\\
\end{align*}
If $n=p(-2pm-\ell_1) - a  + p_1$, then
\begin{align*}
&6n+1 = -(12p^2m + 2p\ell + 6a-p),\qquad\mbox{where $\ell=3\ell_1-1$ and 
$0 < \ell \leq 6p-1$},
\\
&\sin\Lpar{\frac{\pi\,d\, (6n+1)}{p}} 
= \sin\Lpar{-\frac{\pi\,d\, (12p^2m+2p\ell+6a+p)}{p}}
= \sin\Lpar{-\frac{6ad\pi}{p}-\pi d}
\\
&\hspace{104mm}=-(-1)^d\sin\Lpar{\frac{6ad\pi}{p}},
\\
&\sin\Lpar{\frac{\pi}{3}(2\ell+1)} = -\frac{\sqrt{3}}{2}, \quad
(-1)^n = (-1)^{\ell + a + p_1}.\\
\end{align*}
Hence we have
\begin{align*}
\THA{d}{p}{z} 
&= \sum_{\substack{n=-\infty \\ 6n+1\not\equiv 0 \pmod{p}}}^\infty
                   (-1)^n (6n+1) \sin\Lpar{\frac{\pi d (6n+1)}{p}}
\exp\Lpar{3\pi iz\Lpar{n + \frac{1}{6}}^2}\\
&=
(-1)^d\sum_{a=1}^{\frac{1}{2}(p-1)}
\sum_{\ell=0}^{6p-1}
(-1)^{\ell + a + p_1} \, \sin\Lpar{\frac{\pi}{3}(2\ell+1)} \, \frac{2}{\sqrt{3}}
\, \sin\Lpar{\frac{6ad\pi}{p}} \\
&\qquad\qquad \times
\sum_{m=-\infty}^\infty (12p^2m + 2\ell p + 6a + p) \,
\exp\Lpar{\frac{\pi iz}{12}(12p^2m + 2\ell p + 6a+p)^2}\\
&=
(-1)^d\frac{2}{\sqrt{3}}(-1)^{p_1}\sum_{a=1}^{\frac{1}{2}(p-1)}(-1)^{a} \, \sin\Lpar{\frac{6 ad \pi}{p}}
 \\
&\qquad\qquad \times
\sum_{\ell=0}^{6p-1} (-1)^{\ell}\sin\Lpar{\frac{\pi}{3}(2\ell+1)}\ttwid(2\ell p + 6a + p,12p^2;p^2z)\\
&=
(-1)^d\frac{2}{\sqrt{3}} \,\chi_{12}(p) \,
\sum_{a=1}^{\frac{1}{2}(p-1)} (-1)^a 
\sin\Lpar{\frac{6ad\pi}{p}} \,\Theta_1(0,-a,p;p^2z),
\end{align*}
since $(-1)^{p_1} = \chi_{12}(p)$.

\subsubsection*{CASE 2} $p\equiv -1 \pmod{6}$. We proceed as in CASE 1
except this time we let  $p_1 = \frac{1}{6}(p+1)$
so that $6p_1-1=p$, and we find  that
each integer $n$ satisfying $6n+1\not\equiv0 \pmod{p}$ can be written uniquely as
\begin{align*}
&(i) \qquad n=p(2pm+\ell_1) + a  - p_1, \quad
     \mbox{where $1\le a \le \frac{1}{2}(p-1)$, $1 \le \ell_1 \le 2p$, and 
           $m \in \mathbb{Z}$,}\\
&\quad\mbox{or}\\
&(ii) \qquad n=p(-2pm-\ell_1) - a  - p_1, \quad
     \mbox{where $1\le a \le \frac{1}{2}(p-1)$, $0 \le \ell_1 < 2p$, and 
           $m \in \mathbb{Z}$.}
\end{align*}
The result \eqn{theta1id} follows as in CASE 1.
\end{proof}
\noindent
Remark : This result is a generalization of the $d=1$ case given in \cite[Proposition 6.8, p.237]{Ga19a}.
\subsection{Maass-form definitions}
Suppose $0\le a < c$ and $0< b < c$ are integers where $(c,6)=1$.
Following \cite{Ga19a}, we define
\beq
\varepsilon_2\Lpar{\frac{a}{c};z} :=
\begin{cases}
2\,\exp\Lpar{-3\pi iz\Lpar{\frac{a}{c}-\frac{1}{6}}^2} & 
\mbox{if $0<\frac{a}{c}<\frac{1}{6}$},\\
0 & \mbox{if $\frac{1}{6}<\frac{a}{c}<\frac{5}{6}$},\\
2\,\exp\Lpar{-3\pi iz\Lpar{\frac{a}{c}-\frac{5}{6}}^2} & 
\mbox{if $\frac{5}{6}<\frac{a}{c}<1$},
\end{cases}
\label{eq:eps2}
\eeq
\begin{align}
T_1\Lpar{\frac{a}{c};z} &:= -\frac{i}{\sqrt{3}}\int_{-\zcon}^{i\infty}
\frac{\THA{a}{c}{\tau}}{\sqrt{-i(\tau + z)}}\,d\tau, 
\label{eq:T1def}
\\
T_2\Lpar{\frac{a}{c};z} &:= 
\frac{i}{3c}\, 
\int_{-\zcon}^{i\infty} \frac{\Theta_1(0,-a,c;\tau)}{\sqrt{-i(\tau + z)}}\,d\tau.
\label{eq:T2def}
\end{align}
\\
Then the following identity follows from Proposition \propo{theta1id}.
\beq
T_1\Lpar{\frac{d}{p};z}
=
2\,\chi_{12}(p) \,
\sum_{a=1}^{\frac{1}{2}(p-1)} (-1)^{a + d + 1}
\sin\Lpar{\frac{6ad\pi}{p}} \,T_2\Lpar{\frac{a}{p};p^2z}.
\label{eq:T1id}
\eeq
We also define the following two Maass forms of weight $\frac{1}{2}$.
\begin{align}
\mathcal{G}_1\Lpar{\frac{a}{c};z} &:= \Nell{a}{c}{z} - T_1\Lpar{\frac{a}{c};z},
\label{eq:G1acdef}\\
\mathcal{G}_2\Lpar{\frac{a}{c};z} &:= \Mell{a}{c}{z} + \varepsilon_2\Lpar{\frac{a}{c};z}
- T_2\Lpar{\frac{a}{c};z}.
\label{eq:G2acdef}
\end{align}
\subsection{Modular Transformations}

For a function $F(z)$, we define
the usual weight $k$ stroke operator
\beq
 \stroke{F}{A}{k} := (ad-bc)^{k/2}\,(cz+d)^{-k} \, F\left(A\,z\right),\quad\mbox{for}\quad
A=\AMat\in\GL_2^{+}(\mathbb{Z}),
\label{eq:strokedef}
\eeq
where $k\in\frac{1}{2}\mathbb{Z}$, and when calculating $(cz+d)^{-k}$ we take the principal
value. The following theorem that concerns with the transformation of a Maass form under the congruence subgroup $\Gamma_0(p)$ is one of the main results in \cite{Ga19a}.

\begin{theorem}[{\cite[Theorem 4.1, p.218]{Ga19a}}]
\label{thm:mainthm}
Let $p>3$ be prime, suppose $1\le\ell\le(p-1)$, and define
\begin{equation}
\Fell{1}{\ell}{p}{z} := \eta(z) \, \Gell{1}{\ell}{p}{z}.
\label{eq:F1def}
\end{equation}
Then
\beq
\stroke{\Fell{1}{\ell}{p}{z}}{A}{1} = \mu(A,\ell)\,
\Fell{1}{\overline{d\ell}}{p}{z},
\label{eq:Fmult}
\eeq
where
$\mu(A,\ell) = \exp\left(\frac{3\pi i c d \ell^2}{p^2}\right) \,(-1)^{\frac{c\ell}{p}} \, (-1)^{\FL{\frac{d\ell}{p}}},$
and
$A = \AMat \in \Gamma_0(p).$
\\
Here $\overline{m}$ is the least nonnegative residue of $m\pmod{p}$.\\
\end{theorem}
It is well-known that the matrices
$$
S=\SMat,\qquad T=\TMat.
$$
generate $\SL_2(\mathbb{Z})$, and
\beq
\stroke{\eta(z)}{T}{\tfrac{1}{2}} = \zeta_{24} \, \eta(z),\qquad
\stroke{\eta(z)}{S}{\tfrac{1}{2}} = \exp\Lpar{-\tfrac{\pi i}{4}}\, \eta(z).
\label{eq:etatrans}
\eeq
We define
\begin{equation}
\Fell{2}{\ell}{p}{z} := \eta(z) \, \Gell{2}{\ell}{p}{z}.
\label{eq:F2def}
\end{equation}
Then from \eqn{etatrans} and \cite[Theorem 4.5, p.220]{Ga19a}, we have

\beq
\stroke{\Fell{1}{\ell}{p}{z}}{S}{1} = (-i)\, \Fell{2}{\ell}{p}{z},
\label{eq:F1SF2}
\eeq
and
\beq
\Fell{2}{\ell}{p}{z+p} = \zeta_p^{\ell'}\, \Fell{2}{\ell}{p}{z} ,
\label{eq:F2z+p}
\eeq
where $\ell' \equiv  \frac{3}{2}(p-1)\ell^2\pmod{p}$. \\

These transformation identities \eqn{Fmult}, \eqn{F1SF2}, \eqn{F2z+p} will be useful when we examine the transformation of $\mathcal{K}_{p,m}(\zeta_p^d,z)$ in the next section to derive our main result concerning the symmetry of the rank function.\\


\section{Modularity conditions for generalized eta-quotients}

In \cite{Ga19a}, the first author gave an identity for  $\mathcal{K}_{p,0}(\zeta_{p},z)$, for  $p=11,13$, in terms of generalized eta-functions defined in \eqn{Geta}. The proof involves showing that both sides of the identity are weakly holomorphic modular forms of weight $1$ on the appropriate congruence subgroup. In this section, we describe the conditions for a generalized eta-quotient to be a weakly holomorphic modular form of weight $1$ on $\Gamma(p)$. Then in a later section, we derive and prove similar identities for $\mathcal{K}_{p,m}(\zeta_{p},z)$, where $p=11, 13, 17$ and $19$ and $0\leq m \leq p-1$.
\\\\
We present a general criteria for an eta-quotient $j(p,\overrightarrow{n},z)$ (see Definition \refdef{geneta}) to be a weakly holomorphic modular form of weight $1$ on $\Gamma(p)$ in the form of a theorem.
\\
\begin{theorem}\label{thm:modular}
Let $p>3$ be prime and suppose 
$\overrightarrow{n}=(n_0,n_1,n_2,\cdots ,n_{\frac{1}{2}(p-1)}) \in \mathbb{Z}^{\frac{1}{2}(p+1)}$.
Then $j(p,\overrightarrow{n},z)$ is a weakly holomorphic modular form 
of weight 1 on $\Gamma(p)$ satisfying the modularity condition 
$$
j(p,\overrightarrow{n},z)\stroke{}{A}{1}
=\exp(\tfrac{2\pi i b m}{p})j(p,\overrightarrow{n},z)
$$
for $A=\begin{pmatrix} a & b \\ c & d\end{pmatrix} \in \Gamma_1(p)$ 
provided the following conditions are met :
\begin{align*}
&(1)~n_0+\sum_{k=1}^{\frac{1}{2}(p-1)}n_k=2,
\\
&(2)~\sum_{k=1}^{\frac{1}{2}(p-1)}k^2 n_k \equiv 2m \pmod{p},
\\
&(3)~n_0+3\sum_{k=1}^{\frac{1}{2}(p-1)}n_k \equiv 0 \pmod{24}. 
\end{align*}
\end{theorem}

\begin{proof}
The Dedekind eta function is a modular form of weight $\frac{1}{2}$. Thus, $\eta(pz)^{n_0}$ contributes $\frac{n_0}{2}$ and each of the $f_{p,k}(z)^{n_k}$ contributes $\frac{n_k}{2}$ to the weight and the weight of $j(p,\overrightarrow{n},z)$ is $ \frac{n_0}{2}+\sum\limits_{k=1}^{\frac{1}{2}(p-1)}\frac{n_k}{2}$. Condition $(1)$ implies that this weight is $1$.
\\\\\\
Let $A=\begin{pmatrix} a & b \\ c & d\end{pmatrix} \in \Gamma_1(p)$. Then, by \cite[Theorem 6.14, p.243]{Ga19a}, we have
\\
$$
\eta(pz)\stroke{}{A}{1/2}= \nu_\eta({}^{p}A)\, \eta(pz),
$$
where $\nu_\eta({}^{p}A)$ is the eta-multiplier, 
$$
{}^pA = \begin{pmatrix} a & bp \\ c/p & d \end{pmatrix} \in \SL_2(\mathbb{Z}).
$$ 
We note the eta-multiplier $\nu_\eta$ is a 24th root of unity.
Then $$\eta(pAz)=\nu_\eta({}^{p}A)\sqrt{cz+d}~\eta(pz)$$ 
\\
and using the Biagioli transformation \cite[Theorem 6.12, p.243]{Ga19a} for $f_{p,k}(z)$, we have
\begin{align*}
f_{p,k}(z)\stroke{}{A}{1/2}&=(-1)^{k b + \FL{k a/p} + \FL{k/p}}\, \exp\Lpar{\frac{\pi iab}{p}k^2} \,\nu_{\eta}^3\Lpar{{}^pA}\, f_{p,k a}(z)
\\
&=(-1)^{k b + \FL{k a/p}}\, \exp\Lpar{\frac{\pi iab}{p}k^2} \,\nu_{\eta}^3\Lpar{{}^pA}\, f_{p,k}(z),
\end{align*}
assuming $1 \leq k \leq p-1$. Therefore

\begin{align*}
j(p,\overrightarrow{n},z)\stroke{}{A}{1}&=(-1)^{L_1(A)}\exp\Big(\dfrac{\pi i a b}{p}L_2(A)\Big)\,\nu_{\eta}^{L_3(A)}\Lpar{{}^pA}j(z)
\\
&=\exp\Big(\pi i L_1(A)+\dfrac{\pi i a b}{p}L_2(A)\Big)\,\nu_{\eta}^{L_3(A)}\Lpar{{}^pA}j(z),
\end{align*}
where
\begin{align*}
&L_1(A)=b\sum_{k=1}^{\frac{1}{2}(p-1)}kn_k+\sum_{k=1}^{\frac{1}{2}(p-1)}\FL{\frac{ka}{p}} n_k,
\\
&L_2(A)=\sum_{k=1}^{\frac{1}{2}(p-1)}k^2n_k,
\\
&L_3(A)=n_0+3\sum_{k=1}^{\frac{1}{2}(p-1)}n_k.
\end{align*}
\\\\
Now assume conditions $(1)-(3)$ hold, and $\begin{pmatrix} a & b \\ c & d\end{pmatrix} \in \Gamma_1(p)$. Since $L_3(A)\equiv 0 \pmod{24}$, the modularity condition holds if we can show that $L_1(A)+\dfrac{ab}{p}L_2(A)- \dfrac{2bm}{p}$ is an even integer.
\\\\\\
We have that $abL_2(A) \equiv 2bm \pmod{p}$ since $L_2(A) \equiv 2m \pmod{p}$ by $(2)$. Thus $L_1(A)+\dfrac{ab}{p}L_2(A)- \dfrac{2bm}{p}$ is an integer. We show that $L_1(A)+abL_2(A) \equiv 0 \pmod{2}$. This is sufficient to show that it is an even integer. Since $\begin{pmatrix} a & b \\ c & d\end{pmatrix} \in \Gamma_1(p)$, we have $a \equiv 1 \pmod{p}$ and $ka\equiv k \pmod{p}$ so that $ka\equiv p \FL{\frac{ka}{p}}+k$ and $\FL{\frac{ka}{p}} \equiv k(a+1) \pmod{2}$, since $p$ is odd.
\\\\
\begin{align*}
\text{Now,}~ L_1(A) &= \sum_{k=1}^{\frac{1}{2}(p-1)}(bk+\FL{\frac{ka}{p}}) n_k 
\\
&\equiv (a+b+1)\sum_{k=1}^{\frac{1}{2}(p-1)}k n_k \pmod{2}.
\\\\
L_1(A)+abL_2(A) &\equiv (a+b+1)\sum_{k=1}^{\frac{1}{2}(p-1)}kn_k+ab\sum_{k=1}^{\frac{1}{2}(p-1)}k^2 n_k \pmod{2}\\
&\equiv (a+1)(b+1)\sum_{k=1}^{\frac{1}{2}(p-1)}kn_k\equiv 0 \pmod{2}, \\ \text{which always holds since either $a$ or $b$ is odd}.
\end{align*} 

\end{proof}

We also state a lemma which will be of use later.\\
\begin{lemma}\label{lem:eta} For a prime p, 
$$\eta(z)=\eta(pz)^{1-\frac{1}{2}(p-1)}\prod_{k=1}^{\frac{1}{2}(p-1)}f_{p,k}(z).$$
\end{lemma}

\begin{proof}
\begin{align*}
\prod_{n=1}^{\infty}(1-q^n)&=\left(\prod_{k=1}^{p-1}\prod_{n=0}^{\infty}(1-q^{pn+k})\right)\prod_{n=0}^{\infty}(1-q^{pn+p})
\\
&=\left(\prod_{k=1}^{\frac{1}{2}(p-1)}(q^k;q^p)_{\infty}(q^{p-k};q^p)_{\infty}\right)(q^p;q^p)_{\infty}
\\
&=\left(\prod_{k=1}^{\frac{1}{2}(p-1)}f_{p,k}(z)\right)\dfrac{(q^p;q^p)_{\infty}^{\frac{1}{2}(p-1)}}{(q^p;q^p)_{\infty}}.
\end{align*}
This gives the result.\\
\end{proof}

\begin{definition}
Let $\mathfrak{F}(m,p)$ be the set of functions $j(p,\overrightarrow{n},z)$ that satisfy the conditions of Theorem \thm{modular}.\\
\end{definition}

\begin{definition}
\label{def:perm}
Let $p>3$ be prime. For $1\leq r \leq \frac{1}{2}(p-1)$, we define a permutation $\pi_r:[\frac{1}{2}(p-1)]\rightarrow [\frac{1}{2}(p-1)]$, where $[\frac{1}{2}(p-1)]=\{1,2,\cdots,\frac{1}{2}(p-1)\}$ by $\pi_r(i)=i'$ where $ri'\equiv \pm i \pmod{p}$.
\end{definition}
$\pi_r$ induces a permutation on $\mathbb{Z}^{\frac{1}{2}(p-1)}$. For $\overrightarrow{n}=(n_0,n_1,n_2,\cdots ,n_{\frac{1}{2}(p-1)}), \pi_r(\overrightarrow{n})$ permutes the components to $\pi_r(\overrightarrow{n})=(n_0,n_{\pi_r(1)},n_{\pi_r(2)},\cdots ,n_{\pi_r(\frac{1}{2}(p-1))})$.\\

\begin{lemma}\label{lem:jperm}
Let $p>3$ be prime and $\overrightarrow{n}$ and $j(z)$ be defined as in Definition \refdef{geneta}. Then, $$j(p,\pi_r(\overrightarrow{n}),z)=\eta(pz)^{n_0}\prod_{k=1}^{\frac{1}{2}(p-1)}f_{p,rk}(z)^{n_k}.$$
\end{lemma}

\begin{proof} Biagioli's transformation property gives $f_{n,\rho+n}=f_{n,-\rho}=f_{n,\rho}$ \cite[Lemma 2.1, p.278]{Bi89}. Then
\begin{align*}
j(p,\pi_r(\overrightarrow{n}),z)&=\eta(pz)^{n_0}\prod_{k=1}^{\frac{1}{2}(p-1)}f_{p,k}(z)^{n_{\pi_r(k)}}
\\
&=\eta(pz)^{n_0}\prod_{k'=1}^{\frac{1}{2}(p-1)}f_{p,rk'}(z)^{n_{k'}}~~\text{where}~~ rk'\equiv k \pmod{p}
\\
&=\eta(pz)^{n_0}\prod_{k=1}^{\frac{1}{2}(p-1)}f_{p,rk}(z)^{n_{k}}.
\end{align*}
\end{proof}

\begin{theorem}\label{thm:jtrans}
Let $p>3$ be prime, $0\leq m\leq p-1$. Suppose $j(p,\overrightarrow{n},z) \in \mathfrak{F}(m,p)$. Let $A = \begin{pmatrix}
      a & b \\ p & d 
    \end{pmatrix}
\in \Gamma_0(p), 1 \leq a,d \leq p-1
$. Then,

\beq
j(p,\overrightarrow{n},z)\strokeb{A}{1}
=(-1)^{L(\overrightarrow{n},a,b,p)}\,
\exp\Lpar{\frac{2\pi iabm}{p}}\,
  \eta(pz)^{n_0}\prod_{k=1}^{\frac{1}{2}(p-1)}f_{p,ka}(z)^{n_k}, 
\label{eq:jtrans}
\eeq 
where
\beq
L(\overrightarrow{n},a,b,p)=b\,(a+1)\sum\limits_{k=1}^{\frac{1}{2}(p-1)}kn_{k}+\sum\limits_{k=1}^{\frac{1}{2}(p-1)}\FL{\frac{ka}{p}} n_{k}. \label{eq:Lexp}
\eeq 
Also 
\beq
j(p,\pi_r(\overrightarrow{n}),z) \in \mathfrak{F}(m',p),\label{eq:jperm}
\eeq 
where $1\leq r\leq\frac{1}{2}(p-1)$ and $m'\equiv r^2m \pmod p$.
\end{theorem}

\begin{proof}
Following the proof of Theorem \thm{modular}, we use the transformation for $\eta(z)$ and $f_{p,k}(z)$ to get
\begin{align*}
j(p,\overrightarrow{n},z)\stroke{}{A}{1}&=(-1)^{L_1(A)}\exp\Big(\dfrac{\pi i a b}{p}L_2(A)\Big)\,\nu_{\eta}^{L_3(A)}\Lpar{{}^pA}\eta(pz)^{n_0}\prod_{k=1}^{\frac{1}{2}(p-1)}f_{p,ka}(z)^{n_k}
\end{align*}
where
\begin{align*}
&L_1(A)=b\sum_{k=1}^{\frac{1}{2}(p-1)}kn_k+\sum_{k=1}^{\frac{1}{2}(p-1)}\FL{\frac{ka}{p}} n_k,
\\
&L_2(A)=\sum_{k=1}^{\frac{1}{2}(p-1)}k^2n_k,
\\
&L_3(A)=n_0+3\sum_{k=1}^{\frac{1}{2}(p-1)}n_k.
\end{align*}
Since $j(p,\overrightarrow{n},z) \in \mathfrak{F}(m,p)$, we have that 
$L_2(A)\equiv 2\,m \pmod{p}$ and $L_3(A)\equiv 0 \pmod{24}$.  
It suffices to prove that 
$$
pL_{1}(A)+{ab} L_2(A) - p L(A) - 2abm 
$$
is an even multiple of $p$, where
$$
L(A) = 
L(\overrightarrow{n},a,b,p)=b\,(a+1)\sum\limits_{k=1}^{\frac{1}{2}(p-1)}kn_{k}+\sum\limits_{k=1}^{\frac{1}{2}(p-1)}\FL{\frac{ka}{p}} n_{k}. 
$$
Since $L_2(A) \equiv 2\,m\pmod{p}$
we have
$$
pL_{1}(A)+{ab} L_2(A) - p L(A) - 2abm  \equiv 0 \pmod{p}.
$$
Also 
\begin{align*}
pL_{1}(A)+abL_2(A)
    &=b\sum_{k=1}^{\frac{1}{2}(p-1)}kn_{k}
       +\sum_{k=1}^{\frac{1}{2}(p-1)}\FL{\frac{ka}{p}} n_{k}
       +ab\sum_{k=1}^{\frac{1}{2}(p-1)}k^2n_{k}
\\
&\equiv b(1+a)\sum_{k=1}^{\frac{1}{2}(p-1)}kn_{k}
  +\sum_{k=1}^{\frac{1}{2}(p-1)}\FL{\frac{ka}{p}} n_{k} \pmod{2}\\
&\equiv L(A) \pmod{2},
\end{align*}
as required.  Equation \eqn{jtrans} follows.

Now suppose $1 \le r \le \frac{1}{2}(p-1)$ so that for $1 \le i \le \frac{1}{2}(p-1)$
we have
$\pi_r(i) = i'$ where $1 \le i'\le \frac{1}{2}(p-1)$ and $ri' \equiv\pm i \pmod{p}$.
We note that
\begin{align*}
&\sum_{k=1}^{\frac{1}{2}(p-1)}n_k=\sum_{k=1}^{\frac{1}{2}(p-1)}n_{\pi_r(k)},~\text{and}\\
&\sum_{k=1}^{\frac{1}{2}(p-1)}k^2 n_{\pi_r(k)}=\sum_{k=1}^{\frac{1}{2}(p-1)}k^2 n_{k'} \equiv r^2\sum_{k'=1}^{\frac{1}{2}(p-1)}(k^{'})^2n_{k'} \pmod {p}.
\end{align*}
Equation \eqn{jperm} follows easily.
\end{proof}


\section{Main symmetry result}

In this section, we give a proof of our main result illustrating the symmetry of Dyson's rank function. We restate the theorem :

\begin{theorem}
\label{thm:newKpmtrans}
Suppose $p>3$ prime, $0 \le m \le p-1$ and $1 \le d \le p-1$. Let $\mathcal{K}_{p,m}(\zeta_p^d, z)$ be as in Definition \refdef{Kpm}. Then
$$\stroke{\mathcal{K}_{p,m}(\zeta_p, z)}{A}{1}
= \frac{\sin(\pi/p)}{\sin(d\pi/p)}\,(-1)^{d+1}\,
\exp\Lpar{\frac{2\pi imak}{p}}\,\mathcal{K}_{p,ma^2}(\zeta_p^d,z),$$ assuming $1 \le a,d \le (p-1)$ and
$
A = \begin{pmatrix}
      a & k \\ p & d 
    \end{pmatrix}
\in \Gamma_0(p).
$         

\end{theorem}

\subsection{Reformulating the theorem} We need the following to write the theorem in an equivalent form.
\\\\
We assume $p>3$ is prime, and define
\begin{align}      
&\Jdpz{z}{}
\label{eq:Jdef}\\
&= \eta(p^2 z) \,
\left(\Nell{d}{p}{z} - 
2\,\chi_{12}(p) \,
\sum_{\ell=1}^{\frac{1}{2}(p-1)} (-1)^{\ell + d + 1}
\sin\Lpar{\frac{6d\ell\pi}{p}} \,
\left(
\Mell{\ell}{p}{p^2z} + \varepsilon_2\Lpar{\frac{\ell}{p};p^2z}
\right)\right),
\nonumber          
\end{align}
where $\chi_{12}(n)$ is defined in \eqn{chi12}.
Using \eqn{Macid}, \eqn{RNacid} we deduce that

\begin{prop}
\label{propo:RpJ} 
Let $p$ be a prime and $\Rpz{\zeta_p,z}$ be defined as in \eqn{Rpdef}. Then 
$$
\eta(p^2 z) \, \Rpz{\zeta_p^d,z} 
=
\sin\Lpar{\frac{d\pi}{p}} \, \Jdpz{z}{}.
$$
\end{prop}

\begin{proof}
From \eqn{Rpdef} we find that\\
\beq
\Rpz{\zeta_p^d, z} 
:= q^{-\frac{1}{24}} R(\zeta_p^d,q) 
- 
4\,\chi_{12}(p) \,
\sum_{a=1}^{\frac{1}{2}(p-1)} (-1)^{a + d + 1}
 \,
  \sin\Parans{\frac{d\pi}{p}}\,\sin\Parans{\frac{6da\pi}{p}}\,
q^{\tfrac{a}{2}(p-3a)-\tfrac{p^2}{24}}\,\Phi_{p,a}(q^p),
\label{eq:Rpdef2}
\eeq              
for $1 \le d \le (p-1)$.
\\\\
Then
\begin{align*}
&\eta(p^2 z) \, \Rpz{\zeta_p^d,z} \\
&=\eta(p^2 z)\left(
q^{-\frac{1}{24}} R(\zeta_p^d,q) 
- 
4\,\chi_{12}(p) \,
\sum_{a=1}^{\frac{1}{2}(p-1)} (-1)^{a + d + 1}
 \,
  \sin\Parans{\frac{d\pi}{p}}\,\sin\Parans{\frac{6da\pi}{p}}\,
q^{\tfrac{a}{2}(p-3a)-\tfrac{p^2}{24}}\,\Phi_{p,a}(q^p)\right)
\end{align*}
And from the definition of $\Jdpz{z}{}$ in \eqn{Jdef} and from \eqn{Mdef}, we have
\\
\begin{align*}
&\sin\Lpar{\frac{d\pi}{p}} \, \Jdpz{z}{}\\
&=\eta(p^2 z)\left(
\sin\Lpar{\frac{d\pi}{p}}\Nell{d}{p}{z}
 \right.\\
&\qquad \left. - 
2\,\chi_{12}(p) \,
\sum_{a=1}^{\frac{1}{2}(p-1)} (-1)^{a + d + 1}
 \,
  \sin\Parans{\frac{d\pi}{p}}\,\sin\Parans{\frac{6da\pi}{p}}\,
\Lpar{\Mell{a}{p}{p^2z} + \varepsilon_2\Lpar{\frac{a}{p};p^2z}}
\right)\\
&=\eta(p^2 z)\left(
q^{-1/24}\Nac{d}{p}{z}
 \right.\\
&\qquad \left. - 
4\,\chi_{12}(p) \,
\sum_{a=1}^{\frac{1}{2}(p-1)} (-1)^{a + d + 1}
 \,
  \sin\Parans{\frac{d\pi}{p}}\,\sin\Parans{\frac{6da\pi}{p}}\,
\Lpar{q^{3/2a(p-a)-p^2/24}\Mac{a}{p}{p^2z} + \frac{1}{2}
\varepsilon_2\Lpar{\frac{a}{p};p^2z}}
\right).
\end{align*}    
Define
\beq
\twidit{\Phi}_{p,a}(q) 
:= 
\sum_{n=0}^\infty \frac{q^{pn^2}} 
{(q^a;q^p)_{n+1} (q^{p-a};q^p)_n}.
\eeq
Then
$$
\twidit{\Phi}_{p,a}(q^p) 
= 1 + q^{ap} \Mac{a}{p}{p^2z}.
$$
which gives
$$
q^{3/2a(p-a)-p^2/24}\Mac{a}{p}{p^2z}
= q^{a/2(p-3a)-p^2/24}\left(\twidit{\Phi}_{p,a}(q^p)-1\right).
$$
Also from \eqn{eps2}, we have
$$
\frac{1}{2}
\varepsilon_2\Lpar{\frac{a}{p};p^2z}
=
\begin{cases}
q^{a/2(p-3a)-p^2/24} & \mbox{if $0 < 6a < p$}\\
0 & \mbox{if $ p < 6a < 5p$}
\end{cases},
$$
so that
$$
q^{3/2a(p-a)-p^2/24}\Mac{a}{p}{p^2z} + \frac{1}{2}
\varepsilon_2\Lpar{\frac{a}{p};p^2z}
= 
q^{a/2(p-3a)-p^2/24}{\Phi}_{p,a}(q^p).
$$
Then we have
\begin{align*}
&\sin\Lpar{\frac{d\pi}{p}} \, \Jdpz{z}{}\\
&=
\eta(p^2 z)\left(
q^{-1/24}\Nac{d}{p}{z}
- 
4\,\chi_{12}(p) \,
\sum_{a=1}^{\frac{1}{2}(p-1)} (-1)^{a + d + 1}
 \,
  \sin\Parans{\frac{d\pi}{p}}\,\sin\Parans{\frac{6da\pi}{p}}\,
q^{\tfrac{a}{2}(p-3a)-\tfrac{p^2}{24}}\,\Phi_{p,a}(q^p)\right)
\\
&=
\eta(p^2 z) \, \Rpz{\zeta_p^d,z}.
\end{align*}

\end{proof}

\begin{definition}
For $p$ prime, we define the (weight $k$) Atkin $U_p$ operator by
\beq
\stroke{F}{U_p}{k} := \frac{1}{{p}} \sum_{r=0}^{p-1} F\Lpar{\frac{z+r}{p}} 
= p^{\frac{k}{2}-1}\sum_{n=0}^{p-1} \stroke{F}{T_r}{k},
\label{Updef}
\eeq
where
$$T_r = \begin{pmatrix} 1 & r \\ 0 & p \end{pmatrix},$$
and the more general $U_{p,m}$ defined by\\
\beq
\stroke{F}{U_{p,m}}{k} := \frac{1}{p} \sum_{r=0}^{p-1} 
\exp\Lpar{-\frac{2\pi irm}{p}}\,F\Lpar{\frac{z+r}{p}} 
= p^{\frac{k}{2}-1} 
\sum_{r=0}^{p-1} \exp\Lpar{-\frac{2\pi irm}{p}}\, \stroke{F}{T_r}{k}. 
\label{Upkdef}
\eeq\\
\end{definition}
\noindent
We note that $U_p = U_{p,0}$. 
\\\\
In addition, if $F(z) = \sum\limits_n a(n) q^n = \sum\limits_n a(n)\,\exp(2\pi izn)$,
then 
\\
$$\stroke{F}{U_{p,m}}{k} = q^{m/p} \sum\limits_n a(pn+m)\,q^n = \exp(2\pi imz/p)\, \sum\limits_n a(pn+m)\,\exp(2\pi inz).$$

Combining Equation \eqn{Rpdef2} and Proposition \propo{RpJ}, we have
\begin{prop}
\label{propo:Kpm2}
For $p>3$ be a prime and $0 \le m \le p-1$ we have
\beq
\mathcal{K}_{p,m}(\zeta_p^d,z) =  \sin\Lpar{\frac{d\pi}{p}} \, 
\stroke{\Jdpz{z}{}}{U_{p,m}}{1},
\label{eq:Kpmdef}
\eeq
where $\Jdpz{z}{}$ is defined in \eqn{Jdef}.
\end{prop} 

Next we define
\\
\begin{equation}   
\JSdpz{z}{} 
= \frac{\eta(p^2 z)}{\eta(z)} \, \Fell{1}{d}{p}{z}
 - 
2\,\chi_{12}(p) \,
\sum_{\ell=1}^{\frac{1}{2}(p-1)} (-1)^{\ell + d + 1} 
\sin\Lpar{\frac{6\ell d\pi}{p}} \,
\Fell{2}{\ell}{p}{p^2 z}.
\label{eq:JSdef}
\end{equation}   
We have
\begin{prop}
\label{propo:JJSid}
\beq
\Jdpz{z}{} =
\JSdpz{z}{}.
\label{eq:JJSid}
\eeq
\end{prop}
\begin{proof}
Using the definitions in \eqn{F1def}, \eqn{F2def}, \eqn{G1acdef} and \eqn{G2acdef}, we have
\begin{align*}
\JSdpz{z}{} &= \eta(p^2 z)\Lpar{\Gell{1}{d}{p}{z}
 - 
2\,\chi_{12}(p) \,
\sum_{\ell=1}^{\frac{1}{2}(p-1)} (-1)^{\ell + d + 1} 
\sin\Lpar{\frac{6\ell d\pi}{p}} \,
\Gell{2}{\ell}{p}{p^2 z}}\\
&= \eta(p^2 z)\left(\Nell{d}{p}{z} - T_1\Lpar{\frac{d}{p};z}\right.\\
&\quad
 - 
\left.2\,\chi_{12}(p) \,
\sum_{\ell=1}^{\frac{1}{2}(p-1)} (-1)^{\ell + d + 1} 
\sin\Lpar{\frac{6\ell d\pi}{p}} \,
\Lpar{\Mell{\ell}{p}{z} + \varepsilon_2\Lpar{\frac{\ell}{p};z}
- T_2\Lpar{\frac{\ell}{p};z}}\right)\\
&= \eta(p^2 z)\left(\Nell{d}{p}{z} \right.\\
&\quad
 - 
\left.2\,\chi_{12}(p) \,
\sum_{\ell=1}^{\frac{1}{2}(p-1)} (-1)^{\ell + d + 1} 
\sin\Lpar{\frac{6\ell d\pi}{p}} \,
\Lpar{\Mell{\ell}{p}{z} + \varepsilon_2\Lpar{\frac{\ell}{p};z}}
\right) 
\mbox{(by \eqn{T1id})}
\\
&=\Jdpz{z}{}.
\end{align*}
\end{proof}

Thus in view of \eqn{Kpmdef} and \eqn{JJSid} we have the following equivalent form of our main result Theorem \thm{newKpmtrans}.

\begin{theorem}\label{thm:Kpmid}
Let $p>3$ be a prime and $0 \le m \le p-1$. Also, let $1 \le a,d \le (p-1)$ and
$A = \begin{pmatrix}
      a & k \\ p & d 
    \end{pmatrix}
\in \Gamma_0(p).
$ Then with $\Jdpz{z}{}$ and $\JSdpz{z}{}$ as defined in \eqn{Jdef} and \eqn{JSdef} respectively, we have

\beq
\stroke{\JSdpz{z}{}}{U_{p,m}}{1} \strokeb{A}{1}
= 
(-1)^{d+1}\,
\zeta_p^{mak} 
\stroke{\JSdpz{z}{}}{U_{p,ma^2}}{1}.
\label{eq:eqnKpmid}
\eeq\\
\end{theorem}

\subsection{Proof of Theorem \thm{Kpmid}}

We recall from \eqn{JSdef} that
$$                   
\JSdpz{z}{} 
= \frac{\eta(p^2 z)}{\eta(z)} \, \Fell{1}{d}{p}{z}
 - 
2\,\chi_{12}(p) \,
\sum_{j=1}^{\frac{1}{2}(p-1)} (-1)^{j + d + 1} 
\sin\Lpar{\frac{6j d\pi}{p}} \,
\Fell{2}{j}{p}{p^2 z}.
$$
Assume $1 \le a,d \le (p-1)$ 
and
$A = \begin{pmatrix}
      a & k \\ p & d 
    \end{pmatrix}
\in \Gamma_0(p).$ We determine the action of the Atkin operator and matrix $A$ on $\Fell{1}{d}{p}{z}$ and $\Fell{2}{j}{p}{p^2 z}$. We have the following.

\begin{prop}\label{propo:F1transUA}
Let $1 \le a,d \le (p-1)$ and
$A = \begin{pmatrix}
      a & k \\ p & d 
    \end{pmatrix}
\in \Gamma_0(p).
$ Then
$$
\frac{\eta(p^2 z)}{\eta(z)} \, \Fell{1}{1}{p}{z}
\strokeb{U_{p,m}}{1} \strokeb{A}{1}
= (-1)^{d+1} \zeta_p^{mak} 
\frac{\eta(p^2 z)}{\eta(z)} \, \Fell{1}{d}{p}{z}
\strokeb{U_{p,\overline{ma^2}}}{1}.
$$\\
\end{prop}

\begin{proof}
For $0 \le r \le p-1$ let
$T_r = \begin{pmatrix}1 & r \\ 0 & p \end{pmatrix},$
and
$B_r = 
\begin{pmatrix}
a + pr  & \frac{1}{p}( k + rd - r'(a+pr)) \\
p^2 & d - r' p 
\end{pmatrix},$\\
where $0 \le r'\le p-1$ is chosen so that
$r' \equiv r d^2 + dk \pmod{p}.$
Then\\
$$T_r \, A = B_r \, T_{r'},\quad 
r \equiv r' a^2 - ak \pmod{p},\quad\mbox{and}\quad
B_r \in \Gamma_0(p^2).$$\\
We apply Theorem \thm{mainthm} and the well-known result that $\frac{\eta(p^2z)}{\eta(z)}$ is a modular function on $\Gamma_0(p^2)$ when
$p>3$ is prime. We have 
\begin{align*}
\frac{\eta(p^2 z)}{\eta(z)} \, \Fell{1}{1}{p}{z}
  \strokeb{U_{p,m}}{1} \strokeb{A}{1} & = \frac{1}{\sqrt{p}} \sum_{r=0}^{p-1}
  \zeta_p^{-rm} 
  \frac{\eta(p^2 z)}{\eta(z)} \, \Fell{1}{1}{p}{z}
   \strokeb{T_r}{1} \strokeb{A}{1} 
\\
& = \frac{1}{\sqrt{p}} \sum_{r=0}^{p-1}
  \zeta_p^{-rm} 
  \frac{\eta(p^2 z)}{\eta(z)} \, \Fell{1}{1}{p}{z}
   \strokeb{B_r}{1} \strokeb{T_{r'}}{1} 
\\
& = \frac{1}{\sqrt{p}} \sum_{r=0}^{p-1}
  \zeta_p^{-rm} 
  \frac{\eta(p^2 z)}{\eta(z)} \, 
   \mu(B_r,1) 
   \Fell{1}{d}{p}{z}
   \strokeb{T_{r'}}{1} 
\\
& = \frac{1}{\sqrt{p}} (-1)^{d+1} \zeta_p^{mak} \sum_{r'=0}^{p-1}
  \zeta_p^{-r'ma^2} 
  \frac{\eta(p^2 z)}{\eta(z)} \, 
   \Fell{1}{d}{p}{z}
   \strokeb{T_{r'}}{1},
\end{align*}
 
\begin{align*}
&\text{since}\hspace{3mm}\mu(B_r,1) = \exp(\tfrac{3\pi i}{p^2} (p^2(d - pr'))) (-1)^{p}(-1)^{\FL{\frac{d - pr'}{p}}}= (-1)^{d + r' + 1 -r'} = (-1)^{d+1},
\\
&\hspace{12mm}\zeta_p^{-rm} = \zeta_p^{m(-r'a^2 +ak)}
= \zeta_p^{mak} \zeta_p^{-mr'a^2},
\end{align*}
and as $r$ runs through a complete residue system mod $p$ so does
$r'$.
The result follows.
\end{proof}

\begin{prop}\label{propo:F2transU}

Let $1 \le \ell \le \tfrac{1}{2}(p-1)$. Then
$$
\Fell{2}{\ell}{p}{p^2z}
  \strokeb{U_{p,m}}{1} 
=
\begin{cases}
\Fell{2}{\ell}{p}{pz} & \mbox{if $(6\ell)^2 \equiv -24m \pmod{p}$}\\
0 & \mbox{otherwise}
\end{cases}
$$

\end{prop}

\begin{proof}
By \eqn{F2z+p} we have
\begin{align*}
\Fell{2}{\ell}{p}{p^2z} \strokeb{U_{p,m}}{1} & = \frac{1}{p} \sum_{r=0}^{p-1} \zeta_p^{-rm} \Fell{2}{\ell}{p}{pz + pr} 
\\
& = \frac{1}{p} \sum_{r=0}^{p-1} \zeta_p^{-rm + \frac{3r}{2}(p-1)\ell^2} \Fell{2}{\ell}{p}{pz}.
\end{align*}
The result follows since
${-rm + \frac{3r}{2}(p-1)\ell^2}  \equiv 0 \pmod{p}$ if and only if
$(6\ell)^2 \equiv -24m \pmod{p}.$
\end{proof}

\begin{lemma}\label{lem:sineq}
Let $p$ be a prime and $1 \le \ell,\ell' \le \tfrac{1}{2}(p-1)$. Then
$$
\sin\Parans{\frac{6\ell\pi}{p}}
=
(-1)^{\ell'  + a\ell + \FL{a\ell/p}} \sin\Parans{\frac{6\ell'd\pi}{p}}.
$$

\end{lemma}

\begin{proof}
CASE 1. $\ell' \equiv a\ell \pmod{p}$. 
\\\\
Then $a \ell =  p \FL{\frac{a\ell}{p}}+ \ell',$ so that 
$a \ell + \FL{\frac{a\ell}{p}} \equiv  \ell' \pmod{2}.$
\\\\
Also $\ell ' d \equiv ad \ell \equiv \ell \pmod{p},$ so that $\ell ' d  = k' p + \ell,$ for some $k'\in \mathbb{Z}$. 
\\\\
Since $\frac{6\ell ' d \pi}{p}  = 6k'\pi + \frac{6\ell\pi}{p}$, we have
$$
(-1)^{\ell'  + a\ell + \FL{a\ell/p}} \sin\Parans{\frac{6\ell'd\pi}{p}}
=
\sin\Parans{\frac{6\ell'd\pi}{p}}
=
\sin\Parans{\frac{6\ell\pi}{p}}.
$$
\\
CASE 2. $(p -\ell') \equiv a\ell \pmod{p}$. 
\\\\
Then
$a \ell =  p \FL{\frac{a\ell}{p}} + (p  - \ell'),$ so that $a \ell + \FL{\frac{a\ell}{p}} \equiv  \ell' + 1\pmod{2}.$
\\\\
Also $\ell ' d \equiv -ad \ell \equiv -\ell \pmod{p},$ so that $\ell ' d  = k'' p + (p- \ell),$ for some $k''\in \mathbb{Z}$. 
\\\\
Since $\frac{6\ell ' d \pi}{p}  = 6(k''+1)\pi - \frac{6\ell\pi}{p}.$, we have
$$
(-1)^{\ell'  + a\ell + \FL{a\ell/p}} \sin\Parans{\frac{6\ell'd\pi}{p}}
=
-\sin\Parans{\frac{6\ell'd\pi}{p}}
=
\sin\Parans{\frac{6\ell\pi}{p}}.
$$

\end{proof}

\begin{prop}\label{propo:F2transUA}
Let  $1 \le \ell,\ell' \le \tfrac{1}{2}(p-1)$, $1 \le a,d \le (p-1)$
$(6\ell)^2 \equiv -24m \pmod{p}$, and $(6\ell')^2 \equiv -24ma^2 \pmod{p}$. 
Also let
$A = \begin{pmatrix}
      a & k \\ p & d 
    \end{pmatrix}
\in \Gamma_0(p).$ 
Then
\begin{align*}
&(-1)^{\ell} \sin\Parans{\frac{6\ell\pi}{p}}
\Fell{2}{\ell}{p}{p^2z} \strokeb{U_{p,m}}{1} \strokeb{A}{1}= 
(-1)^{\ell'} \sin\Parans{\frac{6\ell'd\pi}{p}}
\zeta_p^{mak} 
\Fell{2}{\ell'}{p}{pz} .
\end{align*}

\end{prop}

\begin{proof}
We let $P = \begin{pmatrix} p & 0 \\ 0 & 1\end{pmatrix},$ and find that $S\, P\, A = B \, S\, P,$
where
$$
B = \begin{pmatrix} d & -1 \\ -pk & a\end{pmatrix}
\in \Gamma_0(p).
$$
From \eqn{F1SF2} we have
$\Fell{2}{\ell}{p}{z}  = i\,\Fell{1}{\ell}{p}{z} \strokeb{S}{1}$.
Then using this and Theorem \thm{mainthm} we have
\begin{align*}
\Fell{2}{\ell}{p}{p^2z} \strokeb{U_{p,m}}{1} \strokeb{A}{1} &=\Fell{2}{\ell}{p}{pz} \strokeb{A}{1}
\\
&=i\,\Fell{1}{\ell}{p}{z} \strokeb{S}{1} \strokeb{P}{1} \strokeb{A}{1}
\\
&=i\,\Fell{1}{\ell}{p}{z} \strokeb{B}{1} \strokeb{S}{1} \strokeb{P}{1}
\\
&=i\,\mu(B,\ell)\, \Fell{1}{\overline{a\ell}}{p}{z} \strokeb{S}{1} \strokeb{P}{1}
\\
&=\mu(B,\ell)\, \Fell{2}{\overline{a\ell}}{p}{pz},
\end{align*}
where
$\mu(B,\ell) = \exp( -\tfrac{3 \pi i}{p} a k \ell^2) (-1)^{k\ell + \FL{a\ell/p}}.$
It can be shown that
$$
\Fell{2}{\ell}{p}{z} 
=
\Fell{2}{p-\ell}{p}{z},
$$
and
$$
(-1)^{\ell} \sin\Parans{\frac{6\ell\pi}{p}}
=
(-1)^{p-\ell} \sin\Parans{\frac{6(p-\ell)\pi}{p}}.
$$
Thus
\begin{align*}
(-1)^{\ell} \sin\Parans{\frac{6\ell\pi}{p}}\Fell{2}{\ell}{p}{p^2z} \strokeb{U_{p,m}}{1} \strokeb{A}{1}
&=(-1)^{\ell} \sin\Parans{\frac{6\ell\pi}{p}}\mu(B,\ell)\, \Fell{2}{\overline{a\ell}}{p}{pz}
\\
&=(-1)^{\ell} \sin\Parans{\frac{6\ell\pi}{p}}\mu(B,\ell)\, \Fell{2}{\ell'}{p}{pz},
\end{align*}
since $\ell' \equiv \pm a\ell \pmod{p}$.\\\\
It remains to show that
\beq
(-1)^{\ell} \sin\Parans{\frac{6\ell\pi}{p}}
	 \mu(B,\ell)
=
(-1)^{\ell'} \sin\Parans{\frac{6\ell'd\pi}{p}}
\zeta_p^{mak}.
\label{eq:muzetaid}
\eeq
\noindent\\
In view of Lemma \lem{sineq} this is equivalent to showing
\beq
(-1)^{\ell  + a\ell + \FL{a\ell/p}} 
	 \mu(B,\ell)
=
\zeta_p^{mak}.
\label{eq:muzetaidB}
\eeq
Substituting for $\mu(B,\ell)$, this is equivalent to

\begin{align*}
&(-1)^{\ell(k + a +1)}
 \exp( -\tfrac{3 \pi i}{p} a k \ell^2) = \zeta_p^{mak},
 \\
\text{or}\hspace{5mm}&p \ell(k + a + 1) - 3ka\ell^2 \equiv 2 m a k \pmod{2p}.
\end{align*}
\noindent
Since $-3\ell^2 \equiv 2m \pmod{p}$ we see that this congruence
holds mod $p$. We also see that it holds mod $2$ trivially when $\ell$ is
even, and holds when $\ell$ is odd, since $(a,k)=1$ so
that
$$
k + a + 1 + ka \equiv (k+1)(a + 1) \equiv 0 \pmod{2},
$$
since either $k$ or $a$ is odd.\\
\end{proof}

We finally combine the above results to give the proof of \eqn{eqnKpmid} using the above results.
\\\\
We consider two cases.
\\\\
CASE 1.  $m=0$ or $\leg{-24m}{p}=-1$.  In this case
$$
(6 \ell)^2 \not\equiv -24m \pmod{p},\quad\mbox{and}\quad
(6 \ell')^2 \not\equiv -24ma^2 \pmod{p},                                 
$$ 
for $1 \le \ell, \ell' \le \tfrac{1}{2}(p-1)$. The result then
follows from Proposition \propo{F1transUA} and Proposition \propo{F2transU}. 
\\\\
CASE 2.  $\leg{-24m}{p}=1$.  In this case choose 
$1 \le \ell, \ell' \le \tfrac{1}{2}(p-1)$ such that          
$$
(6 \ell)^2 \equiv -24m \pmod{p},\quad\mbox{and}\quad
(6 \ell')^2 \equiv -24ma^2 \pmod{p}.                                 
$$
We have
\begin{align*}
&\stroke{\JSdpz{z}{}}{U_{p,m}}{1} \strokeb{A}{1}\\
&= \left(\frac{\eta(p^2 z)}{\eta(z)} \, \Fell{1}{1}{p}{z}
 - 
2\,\chi_{12}(p) \,
\sum_{j=1}^{\frac{1}{2}(p-1)} (-1)^{j} 
\sin\Lpar{\frac{6j\pi}{p}} \,
\Fell{2}{j}{p}{p^2 z} \right)
\strokeb{U_{p,m}}{1} \strokeb{A}{1} \\
&=
(-1)^{d+1} \zeta_p^{mak} 
\frac{\eta(p^2 z)}{\eta(z)} \, \Fell{1}{d}{p}{z}
\strokeb{U_{p,\overline{ma^2}}}{1} \\
& \qquad \qquad  - 
2\,\chi_{12}(p) \,
(-1)^{\ell} 
\sin\Lpar{\frac{6\ell d\pi}{p}} \,
\Fell{2}{\ell}{p}{p^2 z} 
\strokeb{U_{p,m}}{1} \strokeb{A}{1} 
\mbox{(by Propositions \propo{F1transUA} and \propo{F2transU})}
\\
&=
(-1)^{d+1} \zeta_p^{mak} 
\frac{\eta(p^2 z)}{\eta(z)} \, \Fell{1}{d}{p}{z}
\strokeb{U_{p,\overline{ma^2}}}{1} \\
& \qquad \qquad  - 
2\,\chi_{12}(p) \,
(-1)^{\ell'} 
\sin\Lpar{\frac{6\ell' d\pi}{p}} \,
\zeta_p^{mak} \,
\Fell{2}{\ell'}{p}{p z} 
~~\qquad\qquad\mbox{(by Proposition \propo{F2transUA})}
\\
&= (-1)^{d+1} \zeta_p^{mak}
\Bigg(\frac{\eta(p^2 z)}{\eta(z)} \, \Fell{1}{d}{p}{z}
\\ & \qquad \qquad- 
2\,\chi_{12}(p) \,
\sum_{j=1}^{\frac{1}{2}(p-1)} (-1)^{j + d + 1} 
\sin\Lpar{\frac{6jd\pi}{p}} \,
\Fell{2}{j}{p}{p^2 z} \Bigg)
\strokeb{U_{p,ma^2}}{1} 
\mbox{(by Proposition \propo{F1transUA})}\\
&= 
(-1)^{d+1}\,
\zeta_p^{mak} 
\stroke{\JSdpz{d}{z}}{U_{p,ma^2}}{1}.
\end{align*}
This completes the proof.\\
\qed

\begin{theorem}\label{thm:Kpmtransid}
Let $p>3$ be prime, $0\leq m\leq p-1$. Suppose there is a set $\mathcal{B}$ of $\overrightarrow{n}-$vectors such that $\{j(p,\overrightarrow{n},z):\overrightarrow{n} \in \mathcal{B}\}$ is linearly independent (over $\mathbb{Q}$) and

$$\mathcal{K}_{p,m}(\zeta_p,z)=\sum\limits_{\overrightarrow{n} \in \mathcal{B}}c(\overrightarrow{n},\zeta_p)\, j(p,\overrightarrow{n},z),$$ where the $c(\overrightarrow{n},\zeta_p) \in \mathbb{Q}[\zeta_p]$. Then for $1\leq a\leq \frac{1}{2}(p-1)$, we have 
\\
$$\mathcal{K}_{p,\overline{ma^2}}(\zeta_p,z)=\sum\limits_{\overrightarrow{n} \in \mathcal{B}}(-1)^{L(\overrightarrow{n},a,b,p)+d\,(a+1)}\,\frac{\sin(\pi/p)}{\sin(a\pi/p)}\,c(\overrightarrow{n},\zeta_p^a)\, j(p,\pi_a(\overrightarrow{n}),z),$$
where $1 \leq d \leq p-1, ad \equiv 1 \pmod{p}$, and $L(\overrightarrow{n},a,b,p)$ is defined in Equation \eqn{Lexp}.
\end{theorem}

\begin{proof}
From Theorem \thm{Kpmthm}, we have
\beq
\mathcal{K}_{p,\overline{ma^2}}(\zeta_p^d,z)
= \frac{\sin(d\pi/p)}{\sin(\pi/p)}\,(-1)^{d+1}\,
\exp\Lpar{\frac{-2\pi imab}{p}}\,\stroke{\mathcal{K}_{p,m}(\zeta_p, z)}{A}{1},
\nonumber
\eeq
where $b$ is such that $\begin{pmatrix}
      a & b \\ p & d 
    \end{pmatrix}
\in \Gamma_0(p)$ or $ad \equiv 1 \pmod{p}$.
Then using the previous theorem, we get
\begin{align*}
\mathcal{K}_{p,\overline{ma^2}}(\zeta_p^d,z)
&= \sum\limits_{\overrightarrow{n} \in \mathcal{B}} \frac{\sin(d\pi/p)}{\sin(\pi/p)}\,(-1)^{d+1+L(\overrightarrow{n},a,b,p)}\,c(\overrightarrow{n},\zeta_p)\, j(p,\pi_a(\overrightarrow{n}),z)
\\
&= \sum\limits_{\overrightarrow{n} \in \mathcal{B}} \frac{\zeta_p^{\frac{d}{2}}-\zeta_p^{\frac{-d}{2}}}{\zeta_p^{\frac{1}{2}}-\zeta_p^{\frac{-1}{2}}}\,(-1)^{d+1+L(\overrightarrow{n},a,b,p)}\,c(\overrightarrow{n},\zeta_p)\, j(p,\pi_a(\overrightarrow{n}),z)
\end{align*}
Then, replacing $\zeta_p$ by $\zeta_p^a$, and using the fact that $ad=1+pb$, we get
\\
\begin{align*}
\mathcal{K}_{p,\overline{ma^2}}(\zeta_p,z)&=\sum\limits_{\overrightarrow{n} \in \mathcal{B}}\,(-1)^b\,\frac{\sin(\pi/p)}{\sin(a\pi/p)}\,(-1)^{L(\overrightarrow{n},a,b,p)+d+1}\,c(\overrightarrow{n},\zeta_p^a)\, j(p,\pi_a(\overrightarrow{n}),z)
\\
&=\sum\limits_{\overrightarrow{n} \in \mathcal{B}}\,(-1)^{pb}\,\frac{\sin(\pi/p)}{\sin(a\pi/p)}\,(-1)^{L(\overrightarrow{n},a,b,p)+d+1}\,c(\overrightarrow{n},\zeta_p^a)\, j(p,\pi_a(\overrightarrow{n}),z)
\\
&=\sum\limits_{\overrightarrow{n} \in \mathcal{B}}\,(-1)^{L(\overrightarrow{n},a,b,p)+d+ad}\,\frac{\sin(\pi/p)}{\sin(a\pi/p)}\,c(\overrightarrow{n},\zeta_p^a)\, j(p,\pi_a(\overrightarrow{n}),z),
\end{align*}
which proves the theorem.\\
\end{proof}

\begin{cor}\label{cor:resnonres}
Let $p>3$ be prime. The identities for $\mathcal{K}_{p,m}(\zeta_{p},z)$ in terms of generalized eta functions are completely determined by three particular ones namely $$\mathcal{K}_{p,0}(\zeta_{p},z),\, \mathcal{K}_{p,m^+}(\zeta_{p},z), \, \mathcal{K}_{p,m^-}(\zeta_{p},z),$$ where $\left(\frac{-24m^+}{p}\right)=1,\, \left(\frac{-24m^-}{p}\right)=-1$.\\
\end{cor}

We present an example illustrating the theorem when $p=7$. Equation \eqn{Ramid7V2} gives the $7-$dissection of $R(\zeta_7,q)$. We see that $$\mathcal{K}_{7,1}(\zeta_7, z)=- (1+\zeta_7^3 + \zeta_7^4) \, j(7,[3,1,-1,-1],z).$$
\\
Let $\overrightarrow{n}=[3,1,-1,-1], \, p=7, a=2, b=1, d=4, \, c\,(\overrightarrow{n},\zeta_p)=- (1+\zeta_7^3 + \zeta_7^4)$.
\\\\
Now, 
\begin{align*}
L(\overrightarrow{n},a,b,p)&=L([3,1,-1,-1],2,1,7)=-7,
\\
\frac{\sin(\pi/p)}{\sin(a\pi/p)}&=\frac{\sin(\pi/7)}{\sin(2\pi/7)}=1+\zeta_7^2 + \zeta_7^5,
\\
c\,(\overrightarrow{n},\zeta_p)&=c\,([3,1,-1,-1],\zeta_7)=- (1+\zeta_7^3 + \zeta_7^4).
\end{align*}
and 
\begin{align*}
\mathcal{K}_{7,4}(\zeta_7,z)&=\sum\limits_{\overrightarrow{n} \in \mathcal{B}}(-1)^{L(\overrightarrow{n},a,b,p)+a+d}\,\frac{\sin(\pi/p)}{\sin(a\pi/p)}\,c(\overrightarrow{n},\zeta_p^a)\, j(p,\pi_a(\overrightarrow{n}),z)
\\
&=(1+\zeta_7^2 + \zeta_7^5) \,(1+\zeta_7 + \zeta_7^6) \, j(7,[3,-1,1,-1],z)
\\
&=(\zeta_7 + \zeta_7^6) \, j(7,[3,-1,1,-1],z),
\end{align*}
which verifies that 
\begin{align*}
\mathcal{K}_{7,4}(\zeta_7,z)=(\zeta_7 + \zeta_7^6) \, \frac{\eta(7z)^2 \, \eta_{7,2}(z)}{\eta_{7,1}(z)\, \eta_{7,3}(z)}.\\\\
\end{align*}

\section{Symmetry of $\mathcal{K}_{p,0}(\zeta_p,z)$ coefficients}

\begin{theorem}
\label{thm:coeffsymmKp0}
Let $p>3$ be prime. Suppose there are $t$ vectors $\overrightarrow{n_1},\overrightarrow{n_2},\cdots,\overrightarrow{n_t}\in \mathbb{Z}^{\frac{1}{2}(p+1)}$ such that the set of functions $j(p,\pi_r(\overrightarrow{n_{k}}),z), 1\leq {k} \leq t, 1\leq r \leq \frac{1}{2}(p-1)$ are linearly independent (over $\mathbb{Q}$) and $$\mathcal{K}_{p,0}(\zeta_p,z)=\sum_{{k}=1}^t\sum_{r=1}^{\frac{1}{2}(p-1)}c_{p,r,{k}}(\zeta_p)j(p,\pi_r(\overrightarrow{n_{k}}),z),$$ where $c_{p,r,{k}}(\zeta_p) \in \mathbb{Q}[\zeta_p]$. Then for $1\leq {d} \leq \frac{1}{2}(p-1)$, and $\overrightarrow{n}=(n_0,n_1,n_2,\cdots ,n_{\frac{1}{2}(p-1)})$, we have $$c_{p,d,k}(\zeta_p)=\frac{\sin(\pi/p)}{\sin(d\pi/p)}\,(-1)^{d+1+L(\overrightarrow{n},d)}c_{p,1,k}(\zeta_p^d)$$ where $$L(\overrightarrow{n},d)=L(\overrightarrow{n},a,b,d,p)=bd(1+a)\sum\limits_{k=1}^{\frac{1}{2}(p-1)}kn_{k}+\sum\limits_{k=1}^{\frac{1}{2}(p-1)}\Big(\FL{\frac{dka}{p}}+\FL{\frac{dk}{p}}\Big) n_{k},$$ and $a,b$ are chosen so that $A = \begin{pmatrix}
      a & b \\ p & d 
    \end{pmatrix}
\in \Gamma_0(p)$.
\\
\end{theorem}

\begin{proof}
By Theorem \thm{newKpmtrans} (iii), for $m=0$, we have
\\
$$\mathcal{K}_{p,0}(\zeta_p,z)\strokeb{A}{1}=\frac{\sin(\pi/p)}{\sin(d\pi/p)}\,(-1)^{d+1}\mathcal{K}_{p,0}(\zeta_p^d,z)$$ where 
$$A = \begin{pmatrix}
      a & b \\ p & d 
    \end{pmatrix}
\in \Gamma_0(p), 1 \leq a,d \leq p-1.
$$ 
\\
Now as in the proof of Theorem \thm{modular}, we deduce that for an arbitrary $\overrightarrow{n_{\ell}}\in \{\overrightarrow{n_1},\overrightarrow{n_2},\cdots,\overrightarrow{n_t}\}$, say $\overrightarrow{n_{\ell}}=(n_0,n_1,n_2,\cdots ,n_{\frac{1}{2}(p-1)})$, we have 
\\
$j(p,\pi_r(\overrightarrow{n_{\ell}}),z)\stroke{}{A}{1}=(-1)^{L_1(A)+\dfrac{a b}{p}L_2(A)}\,\nu_{\eta}^{L_3(A)}\Lpar{{}^pA}j(p,\pi_{ra}(\overrightarrow{n_{\ell}}),z)$,
\\\\
where
\begin{align*}
&L_{1}(A)=br\sum_{k=1}^{\frac{1}{2}(p-1)}kn_{k}+\sum_{k=1}^{\frac{1}{2}(p-1)}\FL{\frac{rka}{p}} n_{k}+\sum_{k=1}^{\frac{1}{2}(p-1)}\FL{\frac{rk}{p}} n_{k},
\\
&L_{2}(A)=r^2\sum_{k=1}^{\frac{1}{2}(p-1)}k^2n_{k},
\\
&L_{3}(A)=n_0+3\sum_{k=1}^{\frac{1}{2}(p-1)}n_{k}.
\end{align*}
For $m=0$, by Theorem \thm{modular}, we have  $L_{3}(A)\equiv 0 \pmod{24}$ and 
\begin{align*}
L_{1}(A)+abL_2(A)&=br\sum_{k=1}^{\frac{1}{2}(p-1)}kn_{k}+\sum_{k=1}^{\frac{1}{2}(p-1)}\FL{\frac{rka}{p}} n_{k}+\sum_{k=1}^{\frac{1}{2}(p-1)}\FL{\frac{rk}{p}} n_{k}+abr^2\sum_{k=1}^{\frac{1}{2}(p-1)}k^2n_{k}
\\
&\equiv br(1+a)\sum_{k=1}^{\frac{1}{2}(p-1)}kn_{k}+\sum_{k=1}^{\frac{1}{2}(p-1)}\Big(\FL{\frac{rka}{p}}+\FL{\frac{rk}{p}}\Big) n_{k} \pmod{2}.
\end{align*}
Therefore
\\
$j(p,\pi_r(\overrightarrow{n_{\ell}}),z)\stroke{}{A}{1}=(-1)^{L_{r,\ell}(A)}\,j(p,\pi_{ra}(\overrightarrow{n_{\ell}}),z),$ where
$L_{r,\ell}(A)=br(1+a)\sum\limits_{k=1}^{\frac{1}{2}(p-1)}kn_{k}+\sum\limits_{k=1}^{\frac{1}{2}(p-1)}\Big(\FL{\frac{rka}{p}}+\FL{\frac{rk}{p}}\Big) n_{k}$.
\\\\\\
Using the transformation above, we have
\begin{align*}
&\sum_{r=1}^{\frac{1}{2}(p-1)}c_{p,r,{\ell}}(\zeta_p)j(p,\pi_r(\overrightarrow{n_{\ell}}),z)\strokeb{A}{1}=\sum_{r=1}^{\frac{1}{2}(p-1)}c_{p,r,\ell}(\zeta_p)(-1)^{L_{r,\ell}(A)}\,j(p,\pi_{ra}(\overrightarrow{n_{\ell}}),z).\\
\end{align*}

Since $ad\equiv 1(mod~p)$, taking $r\rightarrow dr$ we have 

\begin{align*}
&\sum_{r=1}^{\frac{1}{2}(p-1)}c_{p,dr,\ell}(\zeta_p)(-1)^{L_{dr,\ell}(A)}\,j(p,\pi_{dra}(\overrightarrow{n_{\ell}}),z)=\sum_{r=1}^{\frac{1}{2}(p-1)}c_{p,dr,\ell}(\zeta_p)(-1)^{L_{dr,\ell}(A)}\,j(p,\pi_{r}(\overrightarrow{n_{\ell}}),z).\\
\end{align*}

Thus, comparing the coefficients with $\mathcal{K}_{p,0}(\zeta_p^d,z)$, we have
\begin{align*}
&c_{p,r,\ell}(\zeta_p^d)=\frac{\sin(d\pi/p)}{\sin(\pi/p)}\,(-1)^{d+1}c_{p,dr,\ell}(\zeta_p)(-1)^{L_{dr,\ell}(A)}\,\\
\text{or} \quad & c_{p,d,\ell}(\zeta_p)=\frac{\sin(\pi/p)}{\sin(d\pi/p)}\,(-1)^{d+1+L_{d,\ell}(A)}c_{p,1,\ell}(\zeta_p^d).\\
\end{align*}
\end{proof}

\section{Lower bounds for order of at cusps}

In this section, we calculate lower bounds for the orders of $\mathcal{K}_{p,m}(\zeta_p,z)$ at the cusps of $\Gamma_1(p)$, which we use in proving the $\mathcal{K}_{p,m}(\zeta_p,z)$ identities in the subsequent section.
\\\\

\begin{theorem}
\label{thm:ordlb}
Let $p \geq 3$ be a prime and $0 \leq m \leq p-1$. Then 
\begin{enumerate}
\item[(i)]
$$
\ord(\mathcal{K}_{p,m}(\zeta_p,z);0) 
\begin{cases}
      \geq 0 & \text{if}\ p=5, 7, \\
      =\frac{-1}{24p}(p-5)(p-7) & \text{if}\ p>7;
\end{cases} 
$$
\item[(ii)]
$$
\ord(\mathcal{K}_{p,m}(\zeta_p,z);\frac{1}{n}) 
\begin{cases}
 =\frac{-3}{2p}(\frac{1}{6} (p-1)-n)(\frac{1}{6} (p+1)-n) & \text{if}\ 2\leq n< \frac{1}{6}(p-1), \\
 \geq 0 & otherwise;      
\end{cases}
$$  
\item[(iii)]
$$
\ord(\mathcal{K}_{p,m}(\zeta_p,z);\frac{n}{p}) \geq 
\begin{cases}
      \frac{1}{24p}(p^2-1) & \text{if}~m=0~\text{or}~\leg{-24m}{p}=-1  , \\
    \frac{1}{2}-\frac{3}{2p} & otherwise.      
\end{cases}
$$            
\end{enumerate}
\end{theorem}

\begin{proof}
We define
\beq
\SFell{1}{1}{p}{z} = \frac{\eta(p^2 z)}{\eta(z)}\, \Fell{1}{1}{p}{z}.
\label{eq:SFelldef}
\eeq
By Definition \refdef{Kpm} and Proposition \propo{JJSid}, we have,
\begin{align*}
&\mathcal{K}_{p,m}(\zeta_p,z) 
=  \sin\Lpar{\frac{\pi}{p}} \, 
\stroke{\Jpz{1}{z}}{U_{p,m}}{1}
\\
&=\sin\Lpar{\frac{\pi}{p}} \Big[ \frac{\eta(p^2 z)}{\eta(z)} \, \Fell{1}{1}{p}{z}
 - 
2\,\chi_{12}(p) \,
\sum_{\ell=1}^{\frac{1}{2}(p-1)} (-1)^{\ell} 
\sin\Lpar{\frac{6\ell \pi}{p}} \,
\Fell{2}{\ell}{p}{p^2 z} \stroke{\Big]}{U_{p,m}}{1}
\\
&=\begin{cases}
\sin\Lpar{\frac{\pi}{p}} \, \SFell{1}{1}{p}{z}\stroke{}{U_{p,m}}{1},  \text{if}~m=0~\text{or}~\leg{-24m}{p}=-1  , \\
\\
\sin\Lpar{\frac{\pi}{p}} \Big[  \, \SFell{1}{1}{p}{z}\stroke{}{U_{p,m}}{1}
 - 
2\,\chi_{12}(p) \,
 (-1)^{\ell} 
\sin\Lpar{\frac{6\ell \pi}{p}} \,
\Fell{2}{\ell}{p}{p z}\Big], & \text{if}~\leg{-24m}{p}=1, 
\end{cases} 
\\\\
&\text{where}~1 \le \ell \le \tfrac{1}{2}(p-1),
(6\ell)^2 \equiv -24m\pmod{p}, \text{when}~\leg{-24m}{p}=1.
\end{align*}
Therefore
\begin{align*}
& \mathcal{K}_{p,m}(\zeta_p,z)\\
&=\begin{cases}
\frac{\sin\Lpar{\frac{\pi}{p}}}{p}   \,\sum_{k=0}^{p-1} \zeta_p^{-km}  \SFell{1}{1}{p}{z}\stroke{}{T_k}{1}, \text{if}~m=0~\text{or}~\leg{-24m}{p}=-1  , \\
\\
\frac{\sin\Lpar{\frac{\pi}{p}}}{p}   \,\sum_{k=0}^{p-1} \zeta_p^{-km}  \SFell{1}{1}{p}{z}\stroke{}{T_k}{1}
\\
\qquad- 
2\,\chi_{12}(p) \,
 (-1)^{\ell} 
\sin\Lpar{\frac{\pi}{p}}
\sin\Lpar{\frac{6\ell \pi}{p}} \,
\Fell{2}{\ell}{p}{p z},~\text{if}~\leg{-24m}{p}=1.
\end{cases} 
\end{align*}
\\
We consider $\mathcal{K}_{p,m}(\zeta_p,z) \stroke{}{A}{1}$ and evaluate the order of $\mathcal{K}_{p,m}(\zeta_p;z)$ at the cusps considering suitable $A \in \SL_2(\mathbb{Z})$.
\\\\
Order of $\SFell{1}{1}{p}{z}\stroke{}{T_kA}{1}$ at the cusps is calculated in \cite[Theorem 6.9, p.237-238]{Ga19a} as :
\\\\
$\hord(\SFell{1}{1}{p}{\frac{z+k}{p}};0)=\begin{cases}
      0 & \text{if}\ k\neq 0, \\
      1 & \text{if}\ k=0, p=5, \\
       \frac{-1}{24p}(p-5)(p-7) & \text{if}\ k=0, p>5;
    \end{cases}$
\\\\\\
$\hord(\SFell{1}{1}{p}{\frac{z+k}{p}};\frac{1}{n})\begin{cases}
      = 0 & \text{if}\ nk \not\equiv -1(\textrm{mod}\ p), \\ 
       =\frac{-3}{2p}(\frac{1}{6} (p-1)-n)(\frac{1}{6} (p+1)-n) & \text{if}\ nk \equiv -1(\textrm{mod}\ p), 2\leq n \leq \frac{p-1}{6}, \\
       >0 & \text{if}\ nk \equiv -1(\textrm{mod}\ p), \frac{p-1}{6}< n \leq \frac{1}{2}(p-1);
    \end{cases}$
\\\\\\
$\hord(\SFell{1}{1}{p}{\frac{z+k}{p}};\frac{n}{p})=\frac{p^2-1}{24p}.$
\\\\\\
We now look at $\Fell{2}{\ell}{p}{p z} \stroke{}{A}{1}$ and subsequent lower bounds of order at the cusps.
\\\\
Now, as in \cite{Ga19a}, we examine each cusp $\zeta$ of $\Gamma_1(p)$. We choose $A = \begin{pmatrix} a & b \\ c & d\end{pmatrix} \in \SL_2(\mathbb{Z})$ so that $A(\infty)=\frac{a}{c}=\zeta$. Also, let $P = \begin{pmatrix} p & 0 \\ 0 & 1\end{pmatrix}$.
\\\\\\
(i) $\zeta=0 \Rightarrow a=0,c=1$. Let $S = \begin{pmatrix} 0 & -1 \\ 1 & 0\end{pmatrix} \Rightarrow S(\infty)=0=\zeta$. Then, using \cite[Theorem 4.5, p.220]{Ga19a}, we have 
\\
\begin{align*}
\Fell{2}{\ell}{p}{pz} \strokeb{S}{1}=
 \frac{i}{\sqrt{p}}\,\Fell{1}{\ell}{p}{z} \strokeb{S}{1} \strokeb{P}{1} \strokeb{S}{1}&=\frac{i}{\sqrt{p}}\Fell{1}{\ell}{p}{z} \strokeb{\begin{pmatrix} -1 & 0 \\ 0 & -p\end{pmatrix}}{1}
 \\
 &=\frac{-i}{p}\,\Fell{1}{\ell}{p}{\frac{z}{p}}.
\end{align*}

Therefore
\begin{align*}
\hord(\Fell{2}{\ell}{p}{p z};0)&=\frac{1}{p}.\hord(\Fell{1}{\ell}{p}{z};\infty)=0.
\end{align*}

By considering 
$\displaystyle\min_{0\leq k \leq p-1} \hord\Big(\SFell{1}{1}{p}{\frac{z+k}{p}};0\Big)$, 
we have $$\ord(\mathcal{K}_{p,m}(\zeta_p,z);0) 
\begin{cases}
      \geq 0 & \text{if}\ p=5, 7, \\
      =\frac{-1}{24p}(p-5)(p-7) & \text{if}\ p>7.
\end{cases} $$
\\\\    
(ii) $\zeta=\frac{1}{n},  2 \leq n \leq \frac{1}{2}(p-1)$. Let $A = \begin{pmatrix} 1 & 0 \\ n & 1\end{pmatrix} \Rightarrow A(\infty)=\frac{1}{n}=\zeta$. 
Then, using \cite[Theorem 4.5, p.220]{Ga19a}, we have,
\begin{align*}
\Fell{2}{\ell}{p}{pz} \strokeb{A}{1}
=\frac{i}{\sqrt{p}}\,\Fell{1}{\ell}{p}{z} \strokeb{S}{1} \strokeb{P}{1} \strokeb{A}{1}
&=\frac{i}{\sqrt{p}}\,\Fell{1}{\ell}{p}{z} \strokeb{\begin{pmatrix} -n & -1 \\ p & 0\end{pmatrix}}{1}
\\
&=\frac{i}{\sqrt{p}}\,\Fell{1}{\ell}{p}{z} \strokeb{CN}{1},
\end{align*}
 where
$
C = \begin{pmatrix} -n & \frac{nk'-1}{p} \\ p & -k'\end{pmatrix},
N = \begin{pmatrix} 1 & k' \\ 0 & p\end{pmatrix}$,
and $k'$ is chosen so that $nk'\equiv 1\pmod{p}$ and $C \in \Gamma_0(p)$.
\\\\ 
Then, using \cite[Theorem 4.1, p.218]{Ga19a}, we have 
\begin{align*}
\Fell{2}{\ell}{p}{pz} \strokeb{A}{1}=\frac{i}{\sqrt{p}} \mu(C,\ell)\Fell{1}{\overline{-k' \ell}}{p}{z} \strokeb{N}{1},
\end{align*} 
 where $\mu(C,\ell) = \exp(-\tfrac{3 \pi i}{p}  k' \ell^2) 
(-1)^{\ell + \FL{-k'\ell/p}}$.
Therefore
\begin{align*}
\hord(\Fell{2}{\ell}{p}{p z};\frac{1}{n})&=\frac{1}{p}.\hord(\Fell{1}{\overline{-k' \ell}}{p}{z};\infty)=0.   
\end{align*}
\\
By considering $\min_{0\leq k \leq p-1} \hord\Big(\SFell{1}{1}{p}{\frac{z+k}{p}};\frac{1}{n}\Big)$, we have $$\ord(\mathcal{K}_{p,m}(\zeta_p,z);\frac{1}{n}) \begin{cases}
      =\frac{-3}{2p}(\frac{1}{6} (p-1)-n)(\frac{1}{6} (p+1)-n) & \text{if}\ 2\leq n< \frac{1}{6}(p-1), \\
      \geq 0 & otherwise.      
    \end{cases}$$ 
\\\\
(iii) $\zeta=\frac{n}{p},  2 \leq n \leq \frac{1}{2}(p-1)$. We choose integers $b,d$ so that $A = \begin{pmatrix} n & b \\ p & d\end{pmatrix} \in \SL_2(\mathbb{Z}),$  and $A(\infty)=\frac{n}{p}=\zeta$.
\\
We have seen in the proof of Proposition \propo{F2transUA} that when 
$A = \begin{pmatrix} a & k \\ p & d\end{pmatrix} \in \Gamma_0(p)$, then 
$$
\Fell{2}{\ell}{p}{p z} \stroke{}{A}{1}
=\mu(B,\ell)\, \Fell{2}{\overline{a\ell}}{p}{z}\stroke{}{P}{1},
$$
where
$$
P = \begin{pmatrix} p & 0 \\ 0 & 1\end{pmatrix},\quad
B = \begin{pmatrix} d & -1 \\ -pk & a\end{pmatrix}
\in \Gamma_0(p) \quad\mbox{and}\quad 
\mu(B,\ell) = \exp( -\tfrac{3 \pi i}{p} a k \ell^2) 
(-1)^{k\ell + \FL{a\ell/p}}.
$$
Therefore
\begin{align*}
 \hord(\Fell{2}{\ell}{p}{p z};\frac{n}{p})&=p.\hord(\Fell{2}{\overline{n\ell}}{p}{z};\infty)
\\
&=\begin{cases}
      \frac{{\overline{n\ell}}}{2}-\frac{3({\overline{n\ell}})^2}{2p} & \text{if}\ 1 \leq {\overline{n\ell}} < \frac{p}{6}, \\
      \frac{3{\overline{n\ell}}}{2}-\frac{3({\overline{n\ell}})^2}{2p} & \text{if}\ \frac{p}{6} \leq {\overline{n\ell}} < \frac{5p}{6}, \\
      \frac{5{\overline{n\ell}}}{2}-\frac{3({\overline{n\ell}})^2}{2p}-p & \text{if}\ \frac{5p}{6} \leq {\overline{n\ell}} < p.
    \end{cases}   
\end{align*}
By a calculation, $\hord(\Fell{2}{\ell}{p}{p z};\frac{n}{p}) \geq \frac{1}{2}-\frac{3}{2p}.$
\\\\
It follows that $\ord(\mathcal{K}_{p,m}(\zeta_p,z);\frac{n}{p}) \geq \begin{cases}
      \frac{1}{24p}(p^2-1) & \text{if}~m=0~\text{or}~\leg{-24m}{p}=-1  , \\
      \frac{1}{2}-\frac{3}{2p} & otherwise.      
    \end{cases}
$\\\\
\end{proof}

\section{Rank mod $p$ identities for $p=11, 13, 17$ and $19$}

In this section, we find and prove identities for  
$\mathcal{K}_{p,m}(\zeta_{p},z)$ when $p=11, 13, 17$ and  $19$ in terms 
of generalized eta-functions defined in \eqn{Geta}. Identities of this 
kind were first studied by Atkin and Hussain \cite{At-Hu} for rank mod $11$. 
Rank mod $13$ identities were subsequently considered by O'Brien in his 
thesis \cite{OB-thesis}. The identities for $p=17$ and $19$ are new.
In general, these identities are of the form 
\beq
\mathcal{K}_{p,m}(\zeta_{p},z)= \sum\limits_{k=1}^r c_{p,m,k}\, j_{p,m,k}(z)
\label{eq:Kpmid2}
\eeq
 where $j_{p,m,k}(z)$ are quotients of generalized 
eta-functions, and the $c_{p,m,k}$ are cyclotomic integers.     

In the first part of this section we describe an algorithm for proving
identities of this type using the Valence Formula (Theorem \thm{val} below).
By Theorem \thm{Kpmtransid} (and noted in Corollary \corol{resnonres}),
for each $p$ we need only give identities for three particular cases of 
$-24m\pmod{p}$, corresponding to $0$, a quadratic residue and a 
quadratic non-residue mod $p$.

In Section \subsect{mod11} we give detail of the algorithm for the
case $p=11$. In Section \subsect{mod13} we list the three identities
for $p=13$ but omit detail of the algorithm. In Sections \subsect{mod17}
and \subsect{mod19} we only give the form of the identities for $p=17$ and
$p=19$ omitting the values of the coefficients.

From 
Theorem \thm{KpmthmG}, we know that $\mathcal{K}_{p,m}(\zeta_{p},z)$ is 
a weakly holomorphic modular form of weight $1$ on $\Gamma(p)$.  The proof 
of the identities primarily involves establishing the equality using the 
Valence formula and showing that the RHS is also a weakly holomorphic 
modular form of weight $1$ on $\Gamma(p)$.  To that end we first state 
here the Valence Formula.
 
\begin{theorem}[The Valence Formula \cite{Ra} (p.98)]
\label{thm:val}

Let $f\ne0$ be a modular form of weight $k$ with respect to a subgroup $\Gamma$ of finite index
in $\Gamma(1)=\SL_2(\mathbb{Z})$. Then
\beq
\ORD(f,\Gamma) = \frac{1}{12} \mu \, k,
\label{eq:valform}
\eeq
where $\mu$ is index $\widehat{\Gamma}$ in $\widehat{\Gamma(1)}$,
$$
\ORD(f,\Gamma) := \sum_{\zeta\in R^{*}} \ORD(f,\zeta,\Gamma),
$$
$R^{*}$ is a fundamental region for $\Gamma$,
and
\begin{equation}
\label{eq:ORDdef}
\ORD(f;\zeta;\Gamma) = n(\Gamma;\zeta)\, \ord(f;\zeta),
\end{equation}
for a cusp $\zeta$ and
$n(\Gamma;\zeta)$ denotes the fan width of the cusp $\zeta \pmod{\Gamma}$.
\end{theorem}
\begin{remark}
For $\zeta\in\mathfrak{h}$,
$\ORD(f;\zeta;\Gamma)$ is defined in terms of 
 the invariant order $\ord(f;\zeta)$, which  is interpreted
in the usual sense. See \cite[p.91]{Ra} for details of this and the 
notation used.
\end{remark}

\noindent
{\bf An Algorithm for Proving Rank Mod $p$ Identities.}
We describe an algorithm for proving rank mod $p$ identities that utilizes the
Valence Formula. We apply this algorithm
with the aid of the MAPLE packages THETAIDS and ETA developed by the first author.

We note that in the actual identities deduced and proved in the subsequent subsection, we write them in terms of the permutation $\pi_r$ and the generalized eta-functions $j(p,\pi_r(\overrightarrow{n}),z)$ defined as in Definition \refdef{perm} and \refdef{geneta}. In our algorithm, we write $\mathcal{K}_{p,m}(\zeta_{p},z)$ in terms of generalized eta functions $j_{p,m,k}$, written in a general sense as in \eqn{Kpmid2}, where the sum is finite and the coefficients $c_{p,m,k}$ are nonzero.  

\begin{enumerate}
\item[\textit{Step 1.}] Use Theorem \thm{modular} to check the three modularity conditions for each $j_{p,m,k}(z)$, $1\leq k \leq r$ in the RHS of the expression \eqn{Kpmid2}. This shows that the RHS of \eqn{Kpmid2} is a weakly holomorphic modular form of weight $1$ on $\Gamma(p)$ satisfying the same modularity condition as $\mathcal{K}_{p,m}(\zeta_{p},z)$ in Theorem \thm{KpmthmG}.\\\\
Calculate orders at $i\infty$ of each generalized eta quotient $j_{p,m,k}$.\\

\item[\textit{Step 2.}] For the cases when $m\neq 0$, convert the eta quotients to weight $0$ by dividing each by the eta quotient having the lowest order at $i \infty$ i.e. choose $k=k_0$ such that  
$$
\ORD(j_{p,m,k_0}(z),i\infty,\Gamma_1(p))
=\min\limits_{1 \leq k \leq r} \ORD(j_{p,m,k}(z),i\infty,\Gamma_1(p)).
$$ 
Let $j_0=j_{p,m,k_0}(z)$. Then \eqn{Kpmid2} has the following equivalent form : 
\beq
\frac{\mathcal{K}_{p,m}(\zeta_{p},z)}{j_0}= \sum\limits_{k=1}^r c_{p,m,k}\, \frac{j_{p,m,k}(z)}{j_0} 
\label{eq:Kpmid3}
\eeq\\
We note that the LHS and each term on the RHS is a modular function on $\Gamma_1(p)$.
When $m=0$, we skip this step.\\
\item[\textit{Step 3.}] At each cusp $s \in \mathcal{S}_p$ given in Proposition \propo{cusps1}, calculate 
\begin{align*}
\ORD\left(j_{p,m,k}(z),s,\Gamma_1(p)\right)\quad \text{when} \quad m=0, 
\\
\ORD\left(\frac{j_{p,m,k}(z)}{j_0},s,\Gamma_1(p)\right)\quad \text{when} \quad m\neq 0.
\end{align*}

for $1\leq k \leq r$, using Proposition \propo{ncuspORDS}.

\item[\textit{Step 4.}] At each cusp $s$, calculate the lower bound $\lambda(p,m,s)$ of 
\begin{align*}
\ORD\left(\mathcal{K}_{p,m}(\zeta_{p},z),s,\Gamma_1(p)\right)\quad \text{when} \quad m=0, 
\\
\ORD\left(\frac{\mathcal{K}_{p,m}(\zeta_{p},z)}{j_0},s,\Gamma_1(p)\right)\quad \text{when} \quad m\neq 0.
\end{align*}

using Theorem \thm{ordlb}. We note that value is an integer.

\item[\textit{Step 5.}] Calculate 

 \begin{equation}
\label{eq:BLB1}
    B=
    \begin{cases}
     \sum\limits_{s \in \mathcal{S}_p,s \neq i\infty} \min\left(\left\{\ORD\left(j_{p,m,k}(z),s,\Gamma_1(p)\right):1 \leq k \leq r\right\}\cup \{\lambda(p,m,s)\}\right), & \text{if}\ m=0 
\\
      \sum\limits_{s \in \mathcal{S}_p,s \neq i\infty} \min\left(\left\{\ORD\left(\frac{j_{p,m,k}(z)}{j_0},s,\Gamma_1(p)\right):1 \leq k \leq r\right\}\cup \{\lambda(p,m,s)\}\right), & \text{if}\ m\neq 0
    \end{cases}
  \end{equation}

\item[\textit{Step 6.}] Show 
$$
\ORD(LHS-RHS~\text{of}~\eqn{Kpmid2},i \infty,\Gamma_1(p))\ge -B+1+\frac{\mu}{12}~\text{if}~m=0,
$$ 
$$
\ORD(LHS-RHS~\text{of}~\eqn{Kpmid3},i \infty,\Gamma_1(p))\ge -B+1~\text{if}~m\neq0.
$$
Here 
\begin{equation}
\label{eq:mupeq}
\mu = \frac{1}{2}(p^2-1),
\end{equation}
which is  index of  
of  $\widehat{\Gamma_1(p)}$ in $\widehat{\Gamma(1)}$. 
See \cite[Thm.4.2.5, p.106]{Miyake}.
Then the Valence formula in Theorem \thm{val} implies LHS=RHS and \eqn{Kpmid2} is proved.
\end{enumerate}

To aid with the calculations we include some propositions on cusps and 
orders at cusps.
From \cite[Corollary 4, p.930]{Ch-Ko-Pa} and \cite[Lemma 3, p.929]{Ch-Ko-Pa}
we have
\begin{prop}
\label{propo:cusps1}
Let $p>3$ be prime. Then we have the following set of inequivalent cusps $\mathcal{S}_p$ 
for $\Gamma_1(p)$ and their corresponding fan widths.
$$
\begin{matrix}
\mbox{Cusp:}\quad  & i\infty,\, & 0,\,  & \frac{1}{2},\, & \frac{1}{3},\, & \dots,\, & \frac{1}{\tfrac{1}{2}(p-1)},\, & \frac{2}{p},\, & \frac{3}{p},\, & \dots,\, & \frac{\tfrac{1}{2}(p-1)}{p},\\
\mbox{Fan width:}\quad & 1,\, &  p,\, & p,\, & p,\, & \dots,\, & p,\, & 1,\, & 1,\, & \dots,\, & 1.     
\end{matrix}
$$
\end{prop}
From \cite[Prop.2.1, p.34]{Ko-book} and \cite[Lemma 3.2, p.285]{Bi89} we have
\begin{prop}
\label{propo:ford}
Let $1\le N\nmid \rho$, and $(a,c)=1$. Then
$$
\ord\Lpar{f_{N,\rho}(z),\frac{a}{c}} = 
\frac{g^2}{2N}\Lpar{\frac{a\rho}{g} - \FL{\frac{a\rho}{g}} - \frac{1}{2}}^2,
$$
and
$$
\ord\Lpar{\eta(N z),\frac{a}{c}} = 
\frac{g^2}{2N},
$$
where $g=(N,c)$.
\end{prop}
\begin{remark}
We have corrected an error in the statement of \cite[Prop.6.13]{Ga19a}.
\end{remark}
The following proposition follows easily from \eqn{ORDdef} and 
Propositions \propo{cusps1} and \propo{ford}.
\begin{prop}
\label{propo:ncuspORDS}
Let $p>3$ be prime and suppose $s=\frac{a}{c}$ is one of the cusps listed
in Proposition \propo{cusps1} with $i\infty$ represented by $\frac{1}{p}$.
Then 
\begin{enumerate}
\item[(i)]
If $(c,p)=1$ then
$$
\ORD\left(j(p,\overrightarrow{n},z),s,\Gamma_1(p)\right)
= \frac{1}{24}\left( n_0 + 3 \sum_{j=1}^{(p-1)/2} n_j\right).
$$
\item[(ii)] 
If $c=p$ then
$$
\ORD\left(j(p,\overrightarrow{n},z),s,\Gamma_1(p)\right)
= \frac{p}{24}\left( n_0 + 12 \sum_{j=1}^{(p-1)/2} 
n_j \left( \frac{aj}{p} - \FL{\frac{aj}{p}} - \frac{1}{2}\right)^2\right).
$$
\end{enumerate}
\end{prop}


\subsection{Rank mod 11 identities}
\label{subsec:mod11}

\subsubsection{\textbf{Identity for $\mathcal{K}_{11,0}$}}

Let the permutation $\pi_r$ and the generalized eta function $j(z)=j(p,\overrightarrow{n},z)$ be defined as in Definition \refdef{perm} and \refdef{geneta}. We follow the steps of the stated algorithm in the process of proving the following identity for $\mathcal{K}_{11,0}(\zeta_{11},z)$ :
\\
\begin{align}
\mathcal{K}_{11,0}(\zeta_{11},z)&=(q^{11};q^{11})_\infty \sum_{n=1}^\infty \Lpar{\sum_{k=0}^{10} N(k,11,11n-5)\,\zeta_{11}^k}q^n \label{eq:rank11id0} \\
\nonumber
\\
&=\sum_{r=1}^{5} c_{11,r}\,j(11,\pi_r(\overrightarrow{n_1}),z),
\nonumber
\end{align}
\\
where 
\begin{align*}
\hspace{-43mm}\overrightarrow{n_1}=(15,-4,-2,-3,-2,-2),
\end{align*}
and the coefficients are :\\
\begin{align*}
c_{11,1} &=-(\zeta_{11}^{9}+\zeta_{11}^{8}+2\,\zeta_{11}^{7}+\zeta_{11}^{6}+\zeta_{11}^{5}+2 \,\zeta_{11}^{4}+\zeta_{11}^{3}+\zeta_{11}^{2}+1),
\\
c_{11,2} &=4\,\zeta_{11}^{9}+\zeta_{11}^{8}+2\,\zeta_{11}^{7}+2\,\zeta_{11}^{6}+2\,\zeta_{11}^{5}+2 \,\zeta_{11}^{4}+\zeta_{11}^{3}+4\,\zeta_{11}^{2}+4,
\\
c_{11,3} &=-\zeta_{11}^{9}-2\,\zeta_{11}^{8}+\zeta_{11}^{7}-2\,\zeta_{11}^{6}-2\,\zeta_{11}^{5}+\zeta_{11}^{4}-2\,\zeta_{11}^{3}-\zeta_{11}^{2}-3,
\\
c_{11,4} &=2\,\zeta_{11}^{8}+2\,\zeta_{11}^{7}+2\,\zeta_{11}^{4}+2\,\zeta_{11}^{3}+3,
\\
c_{11,5} &=2\,\zeta_{11}^{9}+2\,\zeta_{11}^{8}+\zeta_{11}^{7}+\zeta_{11}^{4}+2\,\zeta_{11}^{3}+2 \,\zeta_{11}^{2}+1.\\
\end{align*}
We note that a different but similar identity for 
$\mathcal{K}_{11,0}(\zeta_{11},z)$ was found previously by the first
author \cite[Section 6.4]{Ga19a}. 

\begin{enumerate}
\item[Step 1] We check the conditions for modularity as in Theorem \thm{modular} for $j(11,\pi_r(\overrightarrow{n_1}),z), 1\leq r \leq 5$ involved in \eqn{rank11id0}. Here, $p=11, m=0, n_0=15$. By Theorem \thm{jtrans}, we only need to check modularity for $r=1$. With $\overrightarrow{n_1}=(15,-4,-2,-3,-2,-2)$, we easily see that 
$$n_0+\displaystyle \sum_{k=1}^{5}n_k=2,$$ 

$$n_0+3 \displaystyle \sum_{k=1}^{5}n_k = -24 \equiv 0\pmod {24},$$

$$\displaystyle \sum_{k=1}^{5}k^2 n_k = -143 \equiv 0 \pmod {11},$$
as required.
\\\\
Since $m=0$, we skip Step 2.
\\
\item[Step 3] Using Proposition \propo{ncuspORDS},
 we calculate the orders of each of the five functions $f$ at each cusp 
$s$ of 
$\Gamma_1(11)$.

\begin{center}
$\ORD(f,s,\Gamma_1(11))$\\
\begin{tabular}{c c c c c c c c c}\\
&&&& cusp $s$  \\
$f$ & $i\infty$ & $0$ & $1/n$ & $2/11$ & $3/11$ & $4/11$ & $5/11$  \\
$j(11,\pi_1(\overrightarrow{n_1}),z)$ &  $1$  & $-1$ & $-1$  & $2$    & $2$    &  $2$   &  $3$   \\
$j(11,\pi_2(\overrightarrow{n_1}),z)$ &  $2$  & $-1$ & $-1$  & $2$    & $3$    &  $2$   &  $1$   \\
$j(11,\pi_3(\overrightarrow{n_1}),z)$ &  $2$  & $-1$ & $-1$  & $3$    & $2$    &  $1$   &  $2$   \\
$j(11,\pi_4(\overrightarrow{n_1}),z)$ &  $2$  & $-1$ & $-1$  & $2$    & $1$    &  $3$   &  $2$   \\
$j(11,\pi_5(\overrightarrow{n_1}),z)$ &  $3$  & $-1$ & $-1$  & $1$    & $2$    &  $2$   &  $2$   
\end{tabular}
\end{center}
where $2\le n\le 5$.
\item[Step 4] Considering the LHS of equation \eqn{rank11id0}
we now calculate lower bounds $\lambda(11,0,s)$ of 
the  orders 
$\ORD\left(\mathcal{K}_{11,0}(\zeta_{11},z),s,\Gamma_1(11)\right)$ 
for the cusps $s$ of $\Gamma_1(11)$.
\begin{center}
\begin{tabular}{c c c c c}
cusp $s$ &   $\lambda(11,0,s)$\\
$i\infty$    &           \\
$0$          &   $-1$    \\
$1/n$        &   $0$     \\
$2/11$       &   $\CL{5/11}=1$  \\
$3/11$       &   $\CL{5/11}=1$  \\
$4/11$       &   $\CL{5/11}=1$  \\
$5/11$       &   $\CL{5/11}=1$  \\
\end{tabular} 
\end{center} 
where $2\le n\le 5$, using Theorem \thm{ordlb}. 
Again we note that each value is an integer.
\item[Step 5] 
We summarize the calculations in Steps 4 and 5 in a Table.
The gives lower bounds for the LHS and RHS of equation \eqn{rank11id0}
at the cusps $s$.
\begin{center}
\begin{tabular}{c c c c c}
cusp $\zeta$ &  $\ORD(LHS;\zeta)$ & $\ORD(RHS;\zeta)$ 
& $\ORD(LHS-RHS;\zeta)$ \\
$i\infty$     &            &            &             \\
$0$          &  $\ge-1$   &  $\ge-1$   &   $\ge -1$  \\
$1/n$        &  $\ge0$ &  $\ge-1$   &   $\ge -1$  \\
$2/11$        &  $\ge1$    &  $1$       &   $\ge1$    \\ 
$3/11$        &  $\ge1$    &  $1$       &   $\ge 1$   \\ 
$4/11$        &  $\ge1$    &  $1$       &   $\ge 1$   \\ 
$5/11$        &  $\ge1$    &  $1$       &   $\ge 1$   \\ 
\end{tabular} 
\end{center} 
where $2\le n\le 5$. The constant $B$ (in equation \eqn{BLB1}) is
the sum of the lower bounds in the last column, so that $B=-1$.
\item[Step 6] The LHS and RHS are weakly holomorphic modular forms of 
weight $1$ on $\Gamma_1(11)$. So in the Valence Formula, 
$\frac{\mu k}{12} = 5$. 
The result follows provided we can show that 
$\ORD(LHS-RHS,i\infty,\Gamma_1(11))\geq 7$. 
This is easily verified using MAPLE.
\end{enumerate}

\subsubsection{\textbf{A quadratic residue case}}

We follow the steps of the stated algorithm in the process of proving the following identity for $\mathcal{K}_{11,1}(\zeta_{11},z)$ :
\\
\begin{align}
\mathcal{K}_{11,1}(\zeta_{11},z)&=q^\frac{1}{11}(q^{11};q^{11})_\infty \Bigg(\sum_{n=1}^\infty \Lpar{\sum_{k=0}^{10} N(k,11,11n-4)\,\zeta_{11}^k}q^n \label{eq:rank11id1} \\
&+({\zeta_{11}}+{\zeta_{11}}^{10}-{\zeta_{11}}^{9}-{\zeta_{11}}^{2})\,q^{-1} \Phi_{11,5}(q)\Big)
\nonumber
\\
&=\frac{f_{11,5}(z)}{f_{11,1}(z)}\sum_{r=1}^{5} c_{11,r}\,j(11,\pi_r(\overrightarrow{n_1}),z)+\frac{f_{11,4}(z)}{f_{11,5}(z)}\sum_{r=1}^{5} d_{11,r}\,j(11,\pi_r(\overrightarrow{n_1}),z),
\nonumber
\end{align}
\\
where 
\begin{align*}
\hspace{-55mm}\overrightarrow{n_1}=(15,-4,-2,-3,-2,-2),
\end{align*}
and the coefficients are :
\begin{align*}
c_{11,1} &=0,
\\
c_{11,2} &=5\, \zeta_{11}^{9}+\zeta_{11}^{8}+4\, \zeta_{11}^{7}+2\, \zeta_{11}^{6}+2\, \zeta_{11}^{5}+4\, \zeta_{11}^{4}+\zeta_{11}^{3}+5\, \zeta_{11}^{2}+5,
\\
c_{11,3} &=-(5 \zeta_{11}^{9}+3\, \zeta_{11}^{7}+2\, \zeta_{11}^{6}+2\, \zeta_{11}^{5}+3\, \zeta_{11}^{4}+5\, \zeta_{11}^{2}+1),
\\
c_{11,4} &=\zeta_{11}^{9}-\zeta_{11}^{8}-\zeta_{11}^{7}-\zeta_{11}^{4}-\zeta_{11}^{3}++\zeta_{11}^{2}-2,
\\
c_{11,5} &=-(6\, \zeta_{11}^{9}+2\, \zeta_{11}^{8}+3\, \zeta_{11}^{7}+5\, \zeta_{11}^{6}+5\, \zeta_{11}^{5}+3\, \zeta_{11}^{4}+2\, \zeta_{11}^{3}+6\, \zeta_{11}^{2}+5),
\\
d_{11,1} &=0,
\\
d_{11,2} &=0,
\\
d_{11,3} &=\zeta_{11}^{9}+\zeta_{11}^{8}+\zeta_{11}^{6}+\zeta_{11}^{5}+\zeta_{11}^{3}+\zeta_{11}^{2}+1,
\\
d_{11,4} &=0.
\\
d_{11,5} &=-(2\, \zeta_{11}^{9}+\zeta_{11}^{8}+\zeta_{11}^{7}+\zeta_{11}^{6}+\zeta_{11}^{5}+\zeta_{11}^{4}+\zeta_{11}^{3}+2\, \zeta_{11}^{2}+1),
\\
\end{align*}

\begin{enumerate}
\item[Step 1] We check the conditions for modularity as in Theorem \thm{modular} for $\frac{f_{11,5}(z)}{f_{11,1}(z)}j(11,\pi_r(\overrightarrow{n_1}),z)$ and $\frac{f_{11,4}(z)}{f_{11,5}(z)}j(11,\pi_r(\overrightarrow{n_1}),z), 1\leq r \leq 5$ involved in \eqn{rank11id1}. Here, $p=11, m=1, n_0=15$. 
\\
\begin{center}
\begin{tabular}{c c c c c c c c}
generalized eta-functions & $n_1$ & $n_2$ & $n_3$ & $n_4$ & $n_5$ &  $\displaystyle \sum_{k=1}^{5}k^2 n_k$.  \\\\
$\frac{f_{11,5}(z)}{f_{11,1}(z)}j(11,\pi_2(\overrightarrow{n_1}),z)$   & $-3$   &  $-4$   &  $-2$  &  $-2$  &  $-2$  &  $-119$ \\
$\frac{f_{11,5}(z)}{f_{11,1}(z)}j(11,\pi_3(\overrightarrow{n_1}),z)$   & $-3$   &  $-3$   &  $-4$  &  $-2$  &  $-1$  &  $-108$ \\
$\frac{f_{11,5}(z)}{f_{11,1}(z)}j(11,\pi_4(\overrightarrow{n_1}),z)$   & $-4$   &  $-2$   &  $-2$  &  $-4$  &  $-1$  &  $-119$ \\
$\frac{f_{11,5}(z)}{f_{11,1}(z)}j(11,\pi_5(\overrightarrow{n_1}),z)$   & $-3$   &  $-2$   &  $-2$  &  $-3$  &  $-3$  &  $-152$ \\ 
$\frac{f_{11,4}(z)}{f_{11,5}(z)}j(11,\pi_1(\overrightarrow{n_1}),z)$   & $-2$  &  $-3$   &  $-4$  &  $-1$  &  $-3$   &  $-141$ \\ 
$\frac{f_{11,4}(z)}{f_{11,5}(z)}j(11,\pi_2(\overrightarrow{n_1}),z)$   & $-2$  &  $-2$   &  $-2$  &  $-2$  &  $-5$   &  $-185$\\\\
\end{tabular} 
\end{center} 
For each of the generalized eta-functions, we can see that $\displaystyle \sum_{k=1}^{5}k^2 n_k\equiv 2 \pmod {11}, \displaystyle \sum_{k=1}^{5}n_k=-13$. Thus, $n_0+\displaystyle \sum_{k=1}^{5}n_k=2$, and, $n_0+3 \displaystyle \sum_{k=1}^{5}n_k = -24$.
\item[Step 2] Next, we calculate the orders of the generalized eta-functions at $i\infty$ and considering the identity with zero coefficients removed, we find that $k_0=1$. Thus we divide each generalized eta-function by $j_0=\frac{f_{11,5}(z)}{f_{11,1}(z)}j(11,\pi_2(\overrightarrow{n_1}),z)$, which has the lowest order at $i\infty$. 
\item[Step 3] Using Proposition \propo{ncuspORDS},
 we calculate the orders of each of the six functions $f$ at each cusp 
$s$ of 
$\Gamma_1(11)$.
\begin{center}
$\ORD(f,s,\Gamma_1(11))$\\
\begin{tabular}{c c c c c c c c c}\\
 &&&& cusp $s$  \\
$f$ & $i\infty$ & $0$ & $1/n$ & $2/11$ & $3/11$ & $4/11$ & $5/11$  \\
$\frac{f_{11,5}(z)}{f_{11,1}(z)}j(11,\pi_2(\overrightarrow{n_1}),z)/j_0$ &  $1$  & $0$ & $0$   & $0$    & $1$    &  $0$   &  $-2$   \\
$\frac{f_{11,5}(z)}{f_{11,1}(z)}j(11,\pi_3(\overrightarrow{n_1}),z)/j_0$ &  $1$  & $0$ & $0$   & $1$    & $0$    &  $-1$   &  $-1$   \\
$\frac{f_{11,5}(z)}{f_{11,1}(z)}j(11,\pi_4(\overrightarrow{n_1}),z)/j_0$ &  $1$  & $0$ & $0$   & $0$    & $-1$   &  $1$   &  $-1$   \\
$\frac{f_{11,5}(z)}{f_{11,1}(z)}j(11,\pi_5(\overrightarrow{n_1}),z)/j_0$ &  $2$  & $0$ & $0$   & $-1$   & $0$    &  $0$   &  $-1$   \\
$\frac{f_{11,4}(z)}{f_{11,5}(z)}j(11,\pi_1(\overrightarrow{n_1}),z)/j_0$ &  $2$  & $0$ & $0$   & $0$    & $1$    &  $-2$   &  $-1$   \\
$\frac{f_{11,4}(z)}{f_{11,5}(z)}j(11,\pi_2(\overrightarrow{n_1}),z)/j_0$ &  $3$  & $0$ & $0$   & $-2$    & $1$   &  $-1$   &  $-1$               
\end{tabular}
\end{center}
where $2\le n\le 5$.
\item[Step 4] Considering the LHS of equation \eqn{rank11id1}
after division by $j_0$, 
we now calculate lower bounds $\lambda(11,1,s)$ of 
the  orders 
$\ORD\left(\frac{\mathcal{K}_{11,1}(\zeta_{11},z)}{j_0},\zeta,\Gamma_1(11)\right)$
for the cusps $s$ of $\Gamma_1(11)$.
\begin{center}
\begin{tabular}{c c c c c}
cusp $s$ & $\lambda(11,1,\zeta)$     \\
$i\infty$    &                     \\
$0$          &   $0$               \\
$1/n$        &   $1$               \\
$2/11$       &   $-2$  \\
$3/11$       &   $\CL{-16/11}=-1$  \\
$4/11$       &   $\CL{-23/11}=-2$  \\
$5/11$       &   $\CL{-32/11}-2$ 
\end{tabular} 
\end{center} 
where $2\le n\le 5$, using Theorem \thm{ordlb}. 
Again we note that each value is an integer.
\item[Step 5] 
We summarize the calculations in Steps 4 and 5 in a Table.
The gives lower bounds for the LHS and RHS of equation \eqn{rank11id1}
after division by $j_0$,
at the cusps $s$.
\begin{center}
\begin{tabular}{c c c c c}
cusp $s$ & $\ORD(LHS;s)$ & $\ORD(RHS;s)$ 
& $\ORD(LHS-RHS;s)$ \\
$i\infty$   &            &          &            \\
$0$        &  $\ge0$    &  $\ge0$  &   $\ge 0$  \\
$1/n$      &  $\ge1$ &   $\ge0$  &   $\ge 0$  \\
$2/11$      &  $\ge-2$   &  $-2$    &   $\ge -2$ \\ 
$3/11$      &  $\ge-1$   &  $-1$    &   $\ge -1$ \\ 
$4/11$      &  $\ge-2$   &  $-2$    &   $\ge -2$  \\ 
$5/11$      &  $\ge-2$   &  $-2$    &   $\ge -2$ 
\end{tabular} 
\end{center} 
where $2\le n\le 5$. The constant $B$ (in equation \eqn{BLB1}) is
the sum of the lower bounds in the last column, so that $B=-7$.
\item[Step 6] This time the LHS and RHS are weakly holomorphic modular 
forms of weight $0$ on $\Gamma_1(11)$. 
So in the Valence Formula, $\frac{\mu k}{12} = 0$. 
The result follows provided we can show that 
$\ORD(LHS-RHS,i\infty,\Gamma_1(11))\geq 8$. 
This is easily verified using MAPLE.
\end{enumerate}

\subsubsection{\textbf{A quadratic non-residue case}}
We follow the steps of the stated algorithm in the process of proving the following identity for $\mathcal{K}_{11,2}(\zeta_{11},z)$ :
\\
\begin{align}
\mathcal{K}_{11,2}(\zeta_{11},z)&=q^\frac{2}{11}(q^{11};q^{11})_\infty \sum_{n=1}^\infty \Lpar{\sum_{k=0}^{10} N(k,11,11n-3)\,\zeta_{11}^k}q^n \label{eq:rank11id2} \\
\nonumber
\\
&=\frac{f_{11,4}(z)}{f_{11,1}(z)}\sum_{r=1}^{5} c_{11,r}\,j(11,\pi_r(\overrightarrow{n_1}),z)+\frac{f_{11,3}(z)}{f_{11,4}(z)}\sum_{r=1}^{5} d_{11,r}\,j(11,\pi_r(\overrightarrow{n_1}),z),
\nonumber
\end{align}
\\
where 
\begin{align*}
\hspace{-46mm}\overrightarrow{n_1}=(15,-4,-2,-3,-2,-2),
\end{align*}
and the coefficients are :
\begin{align*}
c_{11,1} &=0,
\\
c_{11,2} &=-(4\,\zeta_{11}^{9}+2\, \zeta_{11}^{8}+3\, \zeta_{11}^{7}+3 \,\zeta_{11}^{6}+3\, \zeta_{11}^{5}+3\, \zeta_{11}^{4}+2\, \zeta_{11}^{3}+4\, \zeta_{11}^{2}+6),
\\
c_{11,3} &=-2\, \zeta_{11}^{8}+\zeta_{11}^{7}-\zeta_{11}^{6}-\zeta_{11}^{5}+\zeta_{11}^{4}-2\, \zeta_{11}^{3}-1,
\\
c_{11,4} &=3\, \zeta_{11}^{9}+3\, \zeta_{11}^{8}+\zeta_{11}^{7}+2\, \zeta_{11}^{6}+2\, \zeta_{11}^{5}+\zeta_{11}^{4}+3\, \zeta_{11}^{3}+3\, \zeta_{11}^{2}+6,
\\
c_{11,5} &=-(\zeta_{11}^{9}+\zeta_{11}^{8}+\zeta_{11}^{7}+\zeta_{11}^{4}+\zeta_{11}^{3}+\zeta_{11}^{2}),
\\
d_{11,1} &=0,
\\
d_{11,2} &=\zeta_{11}^{9}+\zeta_{11}^{7}+\zeta_{11}^{4}+\zeta_{11}^{2}+1,
\\
d_{11,3} &=0,
\\
d_{11,4} &=-(\zeta_{11}^8 + \zeta_{11}^7 + \zeta_{11}^4 + \zeta_{11}^3 + 2),
\\
d_{11,5} &=0.\\
\end{align*}

\begin{enumerate}
\item[Step 1] We check the conditions for modularity as in Theorem \thm{modular} for $\frac{f_{11,4}(z)}{f_{11,1}(z)}j(11,\pi_r(\overrightarrow{n_1}),z)$ and $\frac{f_{11,3}(z)}{f_{11,4}(z)}j(11,\pi_r(\overrightarrow{n_1}),z), 1\leq r \leq 5$ involved in \eqn{rank11id2}. Here, $p=11, m=2, n_0=15$. 
\\
\begin{center}
\begin{tabular}{c c c c c c c c}
generalized eta-functions & $n_1$ & $n_2$ & $n_3$ & $n_4$ & $n_5$ &  $\displaystyle \sum_{k=1}^{5}k^2 n_k$.  \\\\
$\frac{f_{11,4}(z)}{f_{11,1}(z)}j(11,\pi_2(\overrightarrow{n_1}),z)$   & $-3$   &  $-4$   &  $-2$  &  $-1$  &  $-3$  &  $-128$ \\
$\frac{f_{11,4}(z)}{f_{11,1}(z)}j(11,\pi_3(\overrightarrow{n_1}),z)$   & $-3$   &  $-3$   &  $-4$  &  $-1$  &  $-2$  &  $-117$ \\
$\frac{f_{11,4}(z)}{f_{11,1}(z)}j(11,\pi_4(\overrightarrow{n_1}),z)$   & $-4$   &  $-2$   &  $-2$  &  $-3$  &  $-2$  &  $-128$ \\
$\frac{f_{11,4}(z)}{f_{11,1}(z)}j(11,\pi_5(\overrightarrow{n_1}),z)$   & $-3$   &  $-2$   &  $-2$  &  $-2$  &  $-4$  &  $-161$ \\ 
$\frac{f_{11,3}(z)}{f_{11,4}(z)}j(11,\pi_1(\overrightarrow{n_1}),z)$   & $-2$  &  $-4$   &  $-1$  &  $-3$  &  $-3$   &  $-150$ \\ 
$\frac{f_{11,3}(z)}{f_{11,4}(z)}j(11,\pi_2(\overrightarrow{n_1}),z)$   & $-3$  &  $-2$   &  $-1$  &  $-5$  &  $-2$   &  $-150$\\\\
\end{tabular} 
\end{center} 
For each of the generalized eta-functions, we can see that $\displaystyle \sum_{k=1}^{5}k^2 n_k\equiv 4 \pmod {11}, \displaystyle \sum_{k=1}^{5}n_k=-13$. Thus, $n_0+\displaystyle \sum_{k=1}^{5}n_k=2$, and, $n_0+3 \displaystyle \sum_{k=1}^{5}n_k = -24$.
\item[Step 2] Next, we calculate the orders of the generalized eta-functions at $i\infty$ and considering the identity with zero coefficients removed, we find that $k_0=1$. Thus we divide each generalized eta-function by $j_0=\frac{f_{11,4}(z)}{f_{11,1}(z)}j(11,\pi_2(\overrightarrow{n_1}),z)$, which has the lowest order at $i\infty$. 
\item[Step 3] Using Proposition \propo{ncuspORDS},
 we calculate the orders of each of the six functions $f$ at each cusp 
$s$ of $\Gamma_1(11)$.
\begin{center}
$\ORD(f ,s,\Gamma_1(11))$\\
\begin{tabular}{c c c c c c c c c}\\
&&&& cusp $s$  \\
$f$ & $i\infty$ & $0$ & $1/n$ & $2/11$ & $3/11$ & $4/11$ & $5/11$  \\
$\frac{f_{11,4}(z)}{f_{11,1}(z)}j(11,\pi_2(\overrightarrow{n_1}),z)/j_0$ &  $1$  & $0$ & $0$   & $0$    & $1$    &  $0$   &  $-2$   \\
$\frac{f_{11,4}(z)}{f_{11,1}(z)}j(11,\pi_3(\overrightarrow{n_1}),z)/j_0$ &  $1$  & $0$ & $0$   & $1$    & $0$    &  $1$   &  $-1$   \\
$\frac{f_{11,4}(z)}{f_{11,1}(z)}j(11,\pi_4(\overrightarrow{n_1}),z)/j_0$ &  $1$  & $0$ & $0$   & $0$    & $-1$   &  $1$   &  $-1$   \\
$\frac{f_{11,4}(z)}{f_{11,1}(z)}j(11,\pi_5(\overrightarrow{n_1}),z)/j_0$ &  $2$  & $0$ & $0$   & $-1$   & $0$    &  $0$   &  $-1$   \\
$\frac{f_{11,3}(z)}{f_{11,4}(z)}j(11,\pi_1(\overrightarrow{n_1}),z)/j_0$ &  $2$  & $0$ & $0$   & $0$    & $0$    &  $1$   &  $-3$   \\
$\frac{f_{11,3}(z)}{f_{11,4}(z)}j(11,\pi_2(\overrightarrow{n_1}),z)/j_0$ &  $2$  & $0$ & $0$   & $0$    & $-2$   &  $2$   &  $-2$               
\end{tabular}
\end{center}
where $2\le n\le 5$.
\item[Step 4] Considering the LHS of equation \eqn{rank11id2}
after division by $j_0$, 
we now calculate lower bounds $\lambda(11,2,s)$ of 
the  orders 
$\ORD\left(\frac{\mathcal{K}_{11,2}(\zeta_{11},z)}{j_0},s,\Gamma_1(11)\right)$
for the cusps $s$ of $\Gamma_1(11)$.
\begin{center}
\begin{tabular}{c c c c c}
cusp $s$ & $\lambda(11,2,\zeta)$     \\
$i\infty$    &                     \\
$0$          &   $0$               \\
$1/n$        &   $1$               \\
$2/11$       &   $\CL{-14/11}=-1$  \\
$3/11$       &   $\CL{-24/11}=-2$  \\
$4/11$       &   $\CL{-16/11}=-1$  \\
$5/11$       &   $\CL{-34/11}=-3$  \\
\end{tabular} 
\end{center} 
where $2\le n\le 5$, using Theorem \thm{ordlb}. 
Again we note that each value is an integer.
\item[Step 5] 
We summarize the calculations in Steps 4 and 5 in a Table.
The gives lower bounds for the LHS and RHS of equation \eqn{rank11id2}
after division by $j_0$,
at the cusps $s$.
\begin{center}
\begin{tabular}{c c c c c}
cusp $s$ & $n(\Gamma_1(11);s)$ & $\ORD(LHS;s)$ & $\ORD(RHS;s)$ 
& $\ORD(LHS-RHS;s)$ \\
$i\infty$    & $1$  &            &          &            \\
$0$          & $11$ &  $\ge0$    &  $\ge0$  &   $\ge 0$  \\
$1/n$        & $11$ &  $\ge1$ &   $\ge0$  &   $\ge 0$  \\
$2/11$       & $1$  &  $\ge-1$   &  $-1$    &   $\ge -1$ \\ 
$3/11$       & $1$  &  $\ge-2$   &  $-2$    &   $\ge -2$ \\ 
$4/11$       & $1$  &  $\ge-1$   &  $-1$    &   $\ge -1$  \\ 
$5/11$       & $1$  &  $\ge-3$   &  $-3$    &   $\ge -3$ \\ 
\end{tabular} 
\end{center} 
where $2\le n\le 5$. The constant $B$ (in equation \eqn{BLB1}) is
the sum of the lower bounds in the last column, so that $B=-7$.
\item[Step 6] As in the previous case, the LHS and RHS are weakly holomorphic 
modular 
forms of weight $0$ on $\Gamma_1(11)$. 
So in the Valence Formula, $\frac{\mu k}{12} = 0$. 
The result follows provided we can show that 
$\ORD(LHS-RHS,i\infty,\Gamma_1(11))\geq 8$. 
This is easily verified using MAPLE.
\end{enumerate}


\subsection{Rank mod 13 identities}
\label{subsec:mod13}

\subsubsection{\textbf{Identity for $\mathcal{K}_{13,0}$}}

Let the permutation $\pi_r$ and the generalized eta function $j(z)=j(p,\overrightarrow{n},z)$ be defined as in Definition \refdef{perm} and \refdef{geneta}. The following is an identity for $\mathcal{K}_{13,0}(\zeta_{13},z)$ in terms of generalized eta-functions :
\\
\begin{align}
\mathcal{K}_{13,0}(\zeta_{13},z)&=(q^{13};q^{13})_\infty \sum_{n=1}^\infty \Lpar{\sum_{k=0}^{12} N(k,13,13n-7)\,\zeta_{13}^k}q^n \label{eq:rank13id1} \\
\nonumber
\\
&=\sum_{k=0}^{1}\sum_{r=1}^{6} \left(\dfrac{\eta(13z)}{\eta(z)}\right)^{2k}c_{13,r,k}\,j(13,\pi_r(\overrightarrow{n_1}),z),
\nonumber
\end{align}
\\
where 
\begin{align*}
\overrightarrow{n_1}=(15,-2,-3,-2,-1,-3,-2),
\end{align*}
and the coefficients are :

\begin{align*}
c_{13,1,0} &=  3\, \zeta_{11}^{11}+3\, \zeta_{11}^{10}+5\, \zeta_{11}^{8}+\zeta_{11}^{7}+\zeta_{11}^{6}+5\, \zeta_{11}^{5}+3\, \zeta_{11}^{3}+3\, \zeta_{11}^{2}+5,
\\
c_{13,2,0} &= \zeta_{11}^{11}-3\, \zeta_{11}^{10}-\zeta_{11}^{9}+2\, \zeta_{11}^{8}-2\, \zeta_{11}^{7}-2\, \zeta_{11}^{6}+2\, \zeta_{11}^{5}-\zeta_{11}^{4}-3\, \zeta_{11}^{3}+\zeta_{11}^{2}-2,
\\
c_{13,3,0} &= 5\, \zeta_{11}^{11}+\zeta_{11}^{10}+5\, \zeta_{11}^{9}+2\, \zeta_{11}^{8}+3\, \zeta_{11}^{7}+3\, \zeta_{11}^{6}+2\, \zeta_{11}^{5}+5\, \zeta_{11}^{4}+\zeta_{11}^{3}+5\, \zeta_{11}^{2}+6,
\\
c_{13,4,0} &= \zeta_{11}^{11}+2\, \zeta_{11}^{10}+2\, \zeta_{11}^{9}+2\, \zeta_{11}^{4}+2\, \zeta_{11}^{3}+\zeta_{11}^{2}-1,
\\
c_{13,5,0} &= -(\zeta_{11}^{11}+\zeta_{11}^{10}+2\, \zeta_{11}^{9}+\zeta_{11}^{8}+2\, \zeta_{11}^{7}+2\, \zeta_{11}^{6}+\zeta_{11}^{5}+2\, \zeta_{11}^{4}+\zeta_{11}^{3}+\zeta_{11}^{2}+1),
\\
c_{13,6,0} &= -\zeta_{11}^{11}+2\, \zeta_{11}^{9}+2\, \zeta_{11}^{8}-\zeta_{11}^{7}-\zeta_{11}^{6}+2\, \zeta_{11}^{5}+2\, \zeta_{11}^{4}-\zeta_{11}^{2}+3,
\\
c_{13,1,1} &=  13\, (\zeta_{11}^{10}-\zeta_{11}^{9}+ \zeta_{11}^{8}+ \zeta_{11}^{5}-\zeta_{11}^{4}+\zeta_{11}^{3}+1),
\\
c_{13,2,1} &=  -13\, \left(\zeta_{11}^{10}+\zeta_{11}^{9}+\zeta_{11}^{7}+\zeta_{11}^{6}+\zeta_{11}^{4}+\zeta_{11}^{3}+2\right),
\\
c_{13,3,1} &= 13\, \left(2\, \zeta_{11}^{11}+\zeta_{11}^{9}+\zeta_{11}^{8}+\zeta_{11}^{7}+\zeta_{11}^{6}+\zeta_{11}^{5}+\zeta_{11}^{4}+2 \zeta_{11}^{2}+2\right),
\\
c_{13,4,1} &= 13\, \left(\zeta_{11}^{11}+\zeta_{11}^{10}+\zeta_{11}^{9}+\zeta_{11}^{8}+\zeta_{11}^{5}+\zeta_{11}^{4}+\zeta_{11}^{3}+\zeta_{11}^{2}+1\right),
\\
c_{13,5,1} &= -13\, \left(\zeta_{11}^{8}+\zeta_{11}^{5}\right),
\\
c_{13,6,1} &= -13\, (\zeta_{11}^{11} + \zeta_{11}^{10} + \zeta_{11}^7 + \zeta_{11}^6 + \zeta_{11}^3 + \zeta_{11}^2).
\end{align*}
We note that a different but similar identity for 
$\mathcal{K}_{13,0}(\zeta_{13},z)$ was found previously by the first
author \cite[Section 6.5]{Ga19a}. 

\subsubsection{\textbf{A quadratic residue case}}

The following is an identity for $\mathcal{K}_{13,2}(\zeta_{13},z)$ in terms of generalized eta-functions : 

\begin{align}
\mathcal{K}_{13,2}(\zeta_{13},z)&=q^\frac{2}{13}(q^{13};q^{13})_\infty \Bigg(\sum_{n=1}^\infty \Lpar{\sum_{k=0}^{12} N(k,13,13n-5)\,\zeta_{13}^k}q^n \label{eq:rank13id2} \\
&-(\zeta_{13}^{7}+\zeta_{13}^{6}-\zeta_{13}^5-\zeta_{13}^8)\,q^{0} \Phi_{13,4}(q)\Bigg)
\nonumber
\\
&=\frac{f_{13,1}(z)}{f_{13,6}(z)}\sum_{k=-1}^{1}\sum_{r=1}^{6} \left(\dfrac{\eta(13z)}{\eta(z)}\right)^{2k}c_{13,r,k}\,j(13,\pi_r(\overrightarrow{n_1}),z),
\nonumber
\end{align}
where 
\begin{align*}
\overrightarrow{n_1}=(15,-2,-3,-2,-1,-3,-2),
\end{align*}
and the coefficients are :

\begin{align*}
c_{13,1,-1} &= \zeta_{13}^{10}-3\, \zeta_{13}^{9}+\zeta_{13}^{8}-2\, \zeta_{13}^{7}-2\, \zeta_{13}^{6}+\zeta_{13}^{5}-3\, \zeta_{13}^{4}+\zeta_{13}^{3}+1,
\\
c_{13,2,-1} &= -(\zeta_{13}^{11}+\zeta_{13}^{10}+2\, \zeta_{13}^{9}+\zeta_{13}^{8}+3\, \zeta_{13}^{7}+3 \,\zeta_{13}^{6}+\zeta_{13}^{5}+2\, \zeta_{13}^{4}+\zeta_{13}^{3}+\zeta_{13}^{2}+1),
\\
c_{13,3,-1} &= -(\zeta_{13}^{11}+\zeta_{13}^{9}+\zeta_{13}^{8}+\zeta_{13}^{5}+\zeta_{13}^{4}+\zeta_{13}^{2}+1),
\\
c_{13,4,-1} &= -\zeta_{13}^{10}-\zeta_{13}^{8}+\zeta_{13}^{7}+\zeta_{13}^{6}-\zeta_{13}^{5}-\zeta_{13}^{3},
\\
c_{13,5,-1} &= -\zeta_{13}^{11}-5\, \zeta_{13}^{10}-7\, \zeta_{13}^{9}-9\, \zeta_{13}^{8}-12\, \zeta_{13}^{7}-12\, \zeta_{13}^{6}-9\, \zeta_{13}^{5}-7\, \zeta_{13}^{4}-5\, \zeta_{13}^{3}-\zeta_{13}^{2}+2,
\\
c_{13,6,-1} &= 4\, \zeta_{13}^{11}+2\, \zeta_{13}^{10}-5\, \zeta_{13}^{9}-3\, \zeta_{13}^{8}+2\, \zeta_{13}^{7}+2\, \zeta_{13}^{6}-3\, \zeta_{13}^{5}-5 \zeta_{13}^{4}+2\, \zeta_{13}^{3}+4\, \zeta_{13}^{2}-3,
\\
c_{13,1,0} &= 0,
\\
c_{13,2,0} &= 0,
\\
c_{13,3,0} &= 0,
\\
c_{13,4,0} &= 0,
\\
c_{13,5,0} &= -(\zeta_{13}^{11}+\zeta_{13}^{10}+2\, \zeta_{13}^{9}+2\, \zeta_{13}^{8}+2\, \zeta_{13}^{7}+2\, \zeta_{13}^{6}+2\, \zeta_{13}^{5}+2\, \zeta_{13}^{4}+\zeta_{13}^{3}+\zeta_{13}^{2}+1),
\\
c_{13,6,0} &= -\zeta_{13}^{8}+\zeta_{13}^{7}+\zeta_{13}^{6}-\zeta_{13}^{5},
\\
c_{13,1,1} &= 13\, (\zeta_{13}^{11}+\zeta_{13}^{10}+\zeta_{13}^{8}+\zeta_{13}^{5}+\zeta_{13}^{3}+\zeta_{13}^{2}+2),
\\
c_{13,2,1} &= -13\, (\zeta_{13}^{10}+\zeta_{13}^{7}+\zeta_{13}^{6}+\zeta_{13}^{3}),
\\
c_{13,3,1} &= -13\, (\zeta_{13}^{11}+\zeta_{13}^{9}+\zeta_{13}^{7}+\zeta_{13}^{6}+\zeta_{13}^{4}+\zeta_{13}^{2}+1),
\\
c_{13,4,1} &= -13\, (\zeta_{13}^{9}+\zeta_{13}^{4}),
\\
c_{13,5,1} &= -13\, (\zeta_{13}^{11}+\zeta_{13}^{10}+\zeta_{13}^{9}+2\zeta_{13}^{8}+2\zeta_{13}^{7}+2\zeta_{13}^{6}+2\zeta_{13}^{5}+\zeta_{13}^{4}+\zeta_{13}^{3}+\zeta_{13}^{2}),
\\
c_{13,6,1} &= 13\, (\zeta_{13}^{11}-\zeta_{13}^{9}-\zeta_{13}^{8}-\zeta_{13}^{5}-\zeta_{13}^{4}+\zeta_{13}^{2}-1).
\end{align*}

\subsubsection{\textbf{A quadratic non-residue case}}

The following is an identity for $\mathcal{K}_{13,1}(\zeta_{13},z)$ in terms of generalized eta-functions : 
\\
\begin{align}
\mathcal{K}_{13,1}(\zeta_{13},z)&=q^\frac{1}{13}(q^{13};q^{13})_\infty \sum_{n=1}^\infty \Lpar{\sum_{k=0}^{12} N(k,13,13n-6)\,\zeta_{13}^k}q^n \label{eq:rank13id3} \\
\nonumber
\\
&=\frac{f_{13,1}(z)}{f_{13,5}(z)}\sum_{k=-1}^{1}\sum_{r=1}^{6} \left(\dfrac{\eta(13z)}{\eta(z)}\right)^{2k}c_{13,r,k}\,j(13,\pi_r(\overrightarrow{n_1}),z),
\nonumber
\end{align}
where 
\begin{align*}
\overrightarrow{n_1}=(15,-2,-3,-2,-1,-3,-2),
\end{align*}
and the coefficients are :

\begin{align*}
c_{13,1,-1} &= \zeta_{13}^{11}-\zeta_{13}^{8}-\zeta_{13}^{5}+\zeta_{13}^{2}+2,
\\
c_{13,2,-1} &= -\zeta_{13}^{11}+4\, \zeta_{13}^{10}+\zeta_{13}^{9}-4\, \zeta_{13}^{8}+3\, \zeta_{13}^{7}+3\, \zeta_{13}^{6}-4\, \zeta_{13}^{5}+\zeta_{13}^{4}+4\, \zeta_{13}^{3}-\zeta_{13}^{2}+6,
\\
c_{13,3,-1} &= 3\, \zeta_{13}^{11}+\zeta_{13}^{10}+2\, \zeta_{13}^{9}+\zeta_{13}^{8}+2\, \zeta_{13}^{7}+2\, \zeta_{13}^{6}+\zeta_{13}^{5}+2\, \zeta_{13}^{4}+\zeta_{13}^{3}+3\, \zeta_{13}^{2}+3,
\\
c_{13,4,-1} &= 2\, \zeta_{13}^{11}+4\, \zeta_{13}^{10}+3\, \zeta_{13}^{9}+2\, \zeta_{13}^{8}-\zeta_{13}^{7}-\zeta_{13}^{6}+2\, \zeta_{13}^{5}+3\, \zeta_{13}^{4}+4\, \zeta_{13}^{3}+2\, \zeta_{13}^{2}-1,
\\
c_{13,5,-1} &= -(3\, \zeta_{13}^{11}+5\, \zeta_{13}^{10}+8\, \zeta_{13}^{9}+10\, \zeta_{13}^{8}+11\, \zeta_{13}^{7}+11\, \zeta_{13}^{6}+10\, \zeta_{13}^{5}+8\, \zeta_{13}^{4}+5\, \zeta_{13}^{3}+3\, \zeta_{13}^{2}+1),
\\
c_{13,6,-1} &= \zeta_{13}^{11}-\zeta_{13}^{9}-2\, \zeta_{13}^{8}-2\, \zeta_{13}^{5}-\zeta_{13}^{4}+\zeta_{13}^{2},
\\
c_{13,1,0} &= 0,
\\
c_{13,2,0} &= 0,
\\
c_{13,3,0} &= 0,
\\
c_{13,4,0} &= 0,
\\
c_{13,5,0} &= \zeta_{13}^{11}-\zeta_{13}^{8}-\zeta_{13}^{5}+\zeta_{13}^{2}+2,
\\
c_{13,6,0} &= 0,
\\
c_{13,1,1} &= -13\, (\zeta_{13}^{11}+\zeta_{13}^{10}+ \zeta_{13}^{9}+ \zeta_{13}^{8}+ \zeta_{13}^{7}+ \zeta_{13}^{6}+ \zeta_{13}^{5}+ \zeta_{13}^{4}+ \zeta_{13}^{3}+ \zeta_{13}^{2}+ 1),
\\
c_{13,2,1} &= 13\, (-\zeta_{13}^{11}+ \zeta_{13}^{10}- \zeta_{13}^{8}-\zeta_{13}^{5}+ \zeta_{13}^{3}- \zeta_{13}^{2}+1),
\\
c_{13,3,1} &= 13\, (\zeta_{13}^{11}+ \zeta_{13}^{9}+ \zeta_{13}^{7}+ \zeta_{13}^{6}+\zeta_{13}^{4}+ \zeta_{13}^{2}+2),
\\
c_{13,4,1} &= 13\, (\zeta_{13}^{11}+ \zeta_{13}^{10}+ \zeta_{13}^{9}+ \zeta_{13}^{4}+ \zeta_{13}^{3}+ \zeta_{13}^{2}),
\\
c_{13,5,1} &= -13\, (\zeta_{13}^{10}+\zeta_{13}^{9}+ \zeta_{13}^{8}+2\, \zeta_{13}^{7}+2\, \zeta_{13}^{6}+ \zeta_{13}^{5}+ \zeta_{13}^{4}+ \zeta_{13}^{3}),
\\
c_{13,6,1} &= -13\, (\zeta_{13}^{9}+\zeta_{13}^{8}+\zeta_{13}^{5}+\zeta_{13}^{4}).
\end{align*}

Below, we present one identity for each of the quadratic residue, quadratic non-residue and $m=0$ cases for $p=17$ and  $p=19$. These are new and do not seem to appear in the literature elsewhere. \\
\subsection{Rank mod 17 identities}
\label{subsec:mod17}

\subsubsection{\textbf{Identity for $\mathcal{K}_{17,0}$}}

Let the permutation $\pi_r$ and the generalized eta function $j(z)=j(p,\overrightarrow{n},z)$ be defined as in Definition \refdef{perm} and \refdef{geneta}. The following is an identity for $\mathcal{K}_{17,0}(\zeta_{17},z)$ in terms of generalized eta-functions :
\\
\begin{align}
\mathcal{K}_{17,0}(\zeta_{17},z)&=(q^{17};q^{17})_\infty \sum_{n=1}^\infty \Lpar{\sum_{k=0}^{16} N(k,17,17n-12)\,\zeta_{17}^k}q^n \label{eq:rank17id1} \\
\nonumber
\\
\nonumber
&=\sum_{k=0}^{2}\sum_{r=1}^{8} \left(\dfrac{\eta(17z)}{\eta(z)}\right)^{3k}c_{17,r,k}\,j(17,\pi_r(\overrightarrow{n_1}),z)\quad+
\\
&\quad\sum_{k=0}^{1}\sum_{r=1}^{8} \left(\dfrac{\eta(17z)}{\eta(z)}\right)^{3k}d_{17,r,k}\,j(17,\pi_r(\overrightarrow{n_2}),z),
\nonumber
\end{align}
\\
where 
\begin{align*}
\overrightarrow{n_1}=(15,-3,-1,-2,-1,-2,-1,-2,-1)
\\
\overrightarrow{n_2}=(27,-2,-2,-3,-2,-4,-4,-4,-4),
\end{align*}
and the coefficients $c_{17,r,k}, d_{17,r,k}$, $1\leq r \leq 8, 0\leq k \leq 2$ are linear combinations of cyclotomic integers like the mod $11$ and $13$ identities found previously. We do not include them here.
\\
\subsubsection{\textbf{A quadratic residue case}}

The following is an identity for $\mathcal{K}_{17,12}(\zeta_{17},z)$ in terms of generalized eta-functions :
\\
\begin{align}
\mathcal{K}_{17,12}(\zeta_{17},z)&=q^\frac{12}{17}(q^{17};q^{17})_\infty \Bigg(\sum_{n=1}^\infty \Lpar{\sum_{k=0}^{16} N(k,17,17n)\,\zeta_{17}^k}q^n \label{eq:rank17id2} \\
&-(\zeta_{17}+\zeta_{17}^{16}-2)\,q^{0} \Phi_{17,3}(q)\Bigg)
\nonumber
\\
&=\frac{f_{17,7}(z)}{f_{17,5}(z)}\Bigg(\sum_{k=-1}^{2}\sum_{r=1}^{8} \left(\dfrac{\eta(17z)}{\eta(z)}\right)^{3k}c_{17,r,k}\,j(17,\pi_r(\overrightarrow{n_1}),z)
\nonumber
\\
&\qquad\qquad+\sum_{k=-1}^{1}\sum_{r=1}^{8} \left(\dfrac{\eta(17z)}{\eta(z)}\right)^{3k}d_{17,r,k}\,j(17,\pi_r(\overrightarrow{n_2}),z)\Bigg),
\nonumber
\\
\nonumber
\end{align}
where
\begin{align*}
&\overrightarrow{n_1}=(15,-3,-1,-2,-1,-2,-1,-2,-1),
\\
&\overrightarrow{n_2}=(27,-2,-2,-3,-2,-4,-4,-4,-4),
\end{align*}
and the coefficients $c_{17,r,k}, d_{17,r,k}$, $1\leq r \leq 8, 0\leq k \leq 2$ are linear combinations of cyclotomic integers like the identities found previously. We do not include them here. We however note that $c_{17,r,-1}=0$ for $1 \leq r \leq 8$ and $r\neq 7$. 
\\\\
\subsubsection{\textbf{A quadratic non-residue case}}
The following is an identity for $\mathcal{K}_{17,1}(\zeta_{17},z)$ in terms of generalized eta-functions :
\\
\begin{align}
\mathcal{K}_{17,1}(\zeta_{17},z)&=q^{\frac{1}{17}}(q^{17};q^{17})_\infty \sum_{n=1}^\infty \Lpar{\sum_{k=0}^{16} N(k,17,17n-11)\,\zeta_{17}^k}q^n \label{eq:rank17id3} \\
\nonumber
\\
&=\frac{f_{17,7}(z)}{f_{17,8}(z)}\Bigg(\sum_{k=0}^{2}\sum_{r=1}^{8} \left(\dfrac{\eta(17z)}{\eta(z)}\right)^{3k}c_{17,r,k}\,j(17,\pi_r(\overrightarrow{n_1}),z)
\nonumber
\\
&\hspace{18mm}+\sum_{k=-1}^{1}\sum_{r=1}^{8} \left(\dfrac{\eta(17z)}{\eta(z)}\right)^{3k}d_{17,r,k}\,j(17,\pi_r(\overrightarrow{n_2}),z)\Bigg),
\nonumber
\end{align}
where
\begin{align*}
&\overrightarrow{n_1}=(15,-3,-1,-2,-1,-2,-1,-2,-1),
\\
&\overrightarrow{n_2}=(27,-2,-2,-3,-2,-4,-4,-4,-4),\\
\end{align*}
and the coefficients $c_{17,r,k}, d_{17,r,k}$, $1\leq r \leq 8, 0\leq k \leq 2$ are linear combinations cyclotomic integers like the identities found previously. We do not include them here. We however note that $d_{17,r,-1}=0$ for $r=3,5,6$ and  $7$.\\\\

\subsection{Rank mod 19 identities}
\label{subsec:mod19}

\subsubsection{\textbf{Identity for $\mathcal{K}_{19,0}$}}

Let the permutation $\pi_r$ and the generalized eta function $j(z)=j(p,\overrightarrow{n},z)$ be defined as in Definition \refdef{perm} and \refdef{geneta}. The following is an identity for $\mathcal{K}_{19,0}(\zeta_{19},z)$ in terms of generalized eta-functions :
\\
\begin{align}
\mathcal{K}_{19,0}(\zeta_{19},z)&=(q^{19};q^{19})_\infty \sum_{n=1}^\infty \Lpar{\sum_{k=0}^{18} N(k,19,19n-15)\,\zeta_{19}^k}q^n \label{eq:rank19id1} \\
\nonumber
\\
\nonumber
&=\sum_{k=0}^{2}\sum_{r=1}^{9} \left(\dfrac{\eta(19z)}{\eta(z)}\right)^{4k}\Big(c_{19,r,k}\,j(19,\pi_r(\overrightarrow{n_1}),z)+d_{19,r,k}\,j(19,\pi_r(\overrightarrow{n_2}),z)+
\\
\nonumber
&\hspace{45mm}e_{19,r,k}\,j(19,\pi_r(\overrightarrow{n_3}),z)\Big),
\end{align}
\\
where 
\begin{align*}
\overrightarrow{n_1}=(27,-3,-2,-4,-4,-3,-3,-2,-3,-1),
\\
\overrightarrow{n_2}=(39,-5,-2,-5,-5,-3,-5,-2,-5,-5),
\\
\overrightarrow{n_3}=(39,-5,-4,-3,-4,-5,-4,-4,-5,-3),
\end{align*}
and the coefficients $c_{19,r,k}, d_{19,r,k}, e_{19,r,k}$, $1\leq r \leq 9, 0\leq k \leq 2$ are linear combinations of cyclotomic integers like the mod $11, 13$ and $17$ identities found previously. We do not include them here.
\\
\subsubsection{\textbf{A quadratic residue case}}

The following is an identity for $\mathcal{K}_{19,15}(\zeta_{19},z)$ in terms of generalized eta-functions :
\\
\begin{align}
\mathcal{K}_{19,15}(\zeta_{19},z)&=q^\frac{15}{19}(q^{19};q^{19})_\infty \Bigg(\sum_{n=1}^\infty \Lpar{\sum_{k=0}^{18} N(k,19,19n)\,\zeta_{19}^k}q^n \label{eq:rank19id2} \\
&\hspace{32mm}+(\zeta_{19}+\zeta_{19}^{18}-2)\,q^{0} \Phi_{19,3}(q)\Bigg)
\nonumber
\\
&=\frac{f_{19,6}(z)}{f_{19,5}(z)}\Bigg(\sum_{k=-1}^{2}\sum_{r=1}^{9} \left(\dfrac{\eta(19z)}{\eta(z)}\right)^{4k}\Big(c_{19,r,k}\,j(19,\pi_r(\overrightarrow{n_1}),z)+d_{19,r,k}\,j(19,\pi_r(\overrightarrow{n_2}),z)
\\
\nonumber
&\hspace{63mm}+e_{19,r,k}\,j(19,\pi_r(\overrightarrow{n_3}),z)\Big)\Bigg),
\nonumber
\\
\nonumber
\end{align}
where\\
\begin{align*}
\overrightarrow{n_1}=(39, -5, -5, -4, -5, -4, -5, -3, -5, -1),
\\
\overrightarrow{n_2}=(39, -3, -5, -5, -5, -3, -2, -4, -5, -5),
\\
\overrightarrow{n_3}=(39, -4, -5, -4, -3, -3, -4, -5, -5, -4),
\end{align*}
and the coefficients $c_{19,r,k}, d_{19,r,k}, e_{19,r,k}$, $1\leq r \leq 9, -1\leq k \leq 2$ are linear combinations of cyclotomic integers like the identities found previously. We do not include them here. We however note that $c_{19,r,2}=0$ for $1 \leq r \leq 9$, $d_{19,r,-1}=0$ for $r=1,2,3,5,6,7,8,9$, and $d_{19,r,0}=0$ for $r=1,2,5,8,9$. 
\\\\
\subsubsection{\textbf{A quadratic non-residue case}}
The following is an identity for $\mathcal{K}_{19,1}(\zeta_{19},z)$ in terms of generalized eta-functions :
\\
\begin{align}
\mathcal{K}_{19,1}(\zeta_{19},z)&=q^{\frac{1}{19}}(q^{19};q^{19})_\infty \sum_{n=1}^\infty \Lpar{\sum_{k=0}^{18} N(k,19,19n-14)\,\zeta_{19}^k}q^n \label{eq:rank19id3} \\
\nonumber
\\
&=\frac{f_{19,8}(z)}{f_{19,9}(z)}\Bigg(\sum_{k=-1}^{2}\sum_{r=1}^{9} \left(\dfrac{\eta(19z)}{\eta(z)}\right)^{4k}\Big(c_{19,r,k}\,j(19,\pi_r(\overrightarrow{n_1}),z)+d_{19,r,k}\,j(19,\pi_r(\overrightarrow{n_2}),z)
\\
\nonumber
&\hspace{63mm}+e_{19,r,k}\,j(19,\pi_r(\overrightarrow{n_3}),z)\Big)\Bigg),
\nonumber
\\
\nonumber
\end{align}
where\\
\begin{align*}
\overrightarrow{n_1}=(39, -5, -5, -4, -5, -4, -5, -3, -5, -1),
\\
\overrightarrow{n_2}=(39, -3, -5, -5, -5, -3, -2, -4, -5, -5),
\\
\overrightarrow{n_3}=(39, -4, -5, -4, -3, -3, -4, -5, -5, -4),\\
\end{align*}
and the coefficients $c_{19,r,k}, d_{19,r,k}, e_{19,r,k}$, $1\leq r \leq 9, -1\leq k \leq 2$ are linear combinations of cyclotomic integers like the identities found previously. We do not include them here. We however note that $c_{19,r,2}=0$ for $1 \leq r \leq 9$, $d_{19,r,-1}=0$ for $1 \leq r \leq 9$, and $d_{19,r,0}=0$ for $r=1,2,4,5,6,7,9$.

\section{Concluding remarks}
\label{sec:conclusion}
In this paper we have found new symmetries for Dyson's rank function.
As well we have extended the work of \cite{Ga19a} and found explicit 
$p$-dissection identities that generalize Ramanujan's result \eqn{Ramid5}
to the cases $p=13$, $17$ and $19$.  It is a non-trivial problem to
find the generalized eta-quotients that are needed in these identities.
What helps is knowing lower bounds for orders at cusps. Our approach
has been by a computer search. It would be interesting to find more
exact conditions for the generalized eta-quotients involved and prove
that identities of this type persist for all larger primes $p$.
The next case to consider is $p=23$.  Another problem to consider
is extending the results of this paper to other rank type functions.
For example investigate whether there are similar symmetry type
results for the $M_2$-rank \cite{Be-Ga02} and for the
overpartition rank \cite{Lo2005}.

\newpage
\bibliographystyle{amsplain}


\end{document}